\documentclass[11pt,a4paper]{article}
\usepackage{amsmath,amsthm}
\usepackage{authblk}
\usepackage{aligned-overset}
\usepackage[style=numeric-comp,maxbibnames=100,bibencoding=utf8,giveninits=true,backend=bibtex]{biblatex}
\usepackage{cancel}
\usepackage{cases}
\usepackage{paralist}
\usepackage[hmargin={26mm,26mm},vmargin={30mm,35mm}]{geometry}
\usepackage{graphicx}
\usepackage{subcaption}
\usepackage{hyperref}
\usepackage[utf8]{inputenc}
\usepackage{newtxtext}
\usepackage[varvw]{newtxmath}
\usepackage{slashed}

\newcommand{\email}[1]{\href{mailto:#1}{#1}}

\theoremstyle{plain}
\newtheorem{thm}{Theorem}[section]

\newtheorem{lem}[thm]{Lemma}
\newtheorem{prop}[thm]{Proposition}
\newtheorem*{lem*}{Lemma} 
\theoremstyle{definition}

\newtheorem{algo}{Algorithm}

\theoremstyle{remark}
\newtheorem{rem}[thm]{Remark}

\newcommand*{\sumT}{\ensuremath{\sum_{T\in\mesh}}}
\newcommand*{\sumS}{\ensuremath{\sum_{S\in\submesh}}}
\newcommand*{\genericVec}{\ensuremath{\boldsymbol \tau}}
\newcommand*{\dx}{\ensuremath{\; \textup d x}}

\newcommand*{\tab}{\hspace*{0.5cm}}

\newcommand*{\IR}{\ensuremath{\mathbb R}}
\newcommand*{\IN}{\ensuremath{\mathbb N}}

\newcommand*{\IP}{\ensuremath{\mathbb P}}

\newcommand*{\mesh}{\ensuremath{\mathcal{T}_h}}
\newcommand*{\facesLoc}{\ensuremath{{\mathcal{F}_T}}}
\newcommand*{\faces}{\ensuremath{{\mathcal{F}_h}}}

\newcommand*{\facesUnion}{\ensuremath{{\Gamma_h}}}
\newcommand*{\normalT}{\ensuremath{\boldsymbol n_{\partial T}}}
\newcommand*{\submesh}{\ensuremath{\mathcal{S}_h}}
\newcommand*{\submeshT}{\ensuremath{\mathcal{S}_T}}
\newcommand*{\subfaces}{\ensuremath{\mathcal{G}_h}}
\newcommand*{\normalS}{\ensuremath{\boldsymbol n_{\partial S}}}

\newcommand*{\projLLoc}{\ensuremath{\Pi_T^{k}}}
\newcommand*{\projLLocM}{\ensuremath{\Pi_T^{\mDeg}}}
\newcommand*{\projLLocN}{\ensuremath{\Pi_T^{\nDeg}}}

\newcommand*{\projLLocVec}{\ensuremath{\boldsymbol \Pi_T^{k}}}
\newcommand*{\projLLocVecM}{\ensuremath{\boldsymbol \Pi_T^{\mDeg}}}

\newcommand*{\projLTraceLoc}{\ensuremath{\Pi_{\facesLoc}^{k}}}

\newcommand*{\projLTraceLocM}{\ensuremath{\Pi_{\facesLoc}^{\mDeg}}}
\newcommand*{\projLGlob}{\ensuremath{\Pi_{\mesh}^{k}}}
\newcommand*{\projLTraceGlob}{\ensuremath{\Pi_{\faces}^{k}}}
\newcommand*{\projLTraceGlobArbitraryDeg}{\ensuremath{\Pi_{\faces}}}
\newcommand*{\projLTraceLocArbitraryDeg}{\ensuremath{\Pi_{\facesLoc}}}
\newcommand*{\projEllipticLoc}{\ensuremath{\varpi_T^{k}}}
\newcommand*{\projEllipticLocPlus}{\ensuremath{\varpi_T^{k+1}}}
\newcommand*{\projLGlobM}{\ensuremath{\Pi_{\mesh}^{\mDeg}}}
\newcommand*{\projLTraceGlobN}{\ensuremath{\Pi_{\faces}^{\nDeg}}}
\newcommand*{\projLTraceGlobM}{\ensuremath{\Pi_{\faces}^{\mDeg}}}
\newcommand*{\projLTraceLocSArbitraryDeg}{\ensuremath{\Pi_{\mathcal F_{T_S}}}}

\newcommand*{\hhoSpaceLoc}{\ensuremath{\mathcal{U}^k_T \times \mathcal{M}^k_T}}
\newcommand*{\hhoSpaceGlob}{\ensuremath{\mathcal{U}^k_h\times \mathcal{M}^k_h}}
\newcommand*{\hhoSpaceTrace}{\ensuremath{\mathcal{M}}}
\newcommand*{\hhoSpaceCell}{\ensuremath{\mathcal{U}}}
\newcommand*{\trialGlob}{\ensuremath{(u_h, \mu_h)}}
\newcommand*{\testGlob}{\ensuremath{(v_h, \nu_h)}}
\newcommand*{\trialLoc}{\ensuremath{(u_T, \mu_T)}}
\newcommand*{\testLoc}{\ensuremath{(v_T, \nu_T)}}
\newcommand*{\trialLocCell}{\ensuremath{u_T}}
\newcommand*{\trialLocTrace}{\ensuremath{\mu_T}}
\newcommand*{\testLocCell}{\ensuremath{v_T}}
\newcommand*{\testLocTrace}{\ensuremath{\nu_T}}
\newcommand*{\testLocTraceS}{\ensuremath{\nu_{T_S}}}
\newcommand*{\trialGlobCell}{\ensuremath{u_h}}
\newcommand*{\trialGlobTrace}{\ensuremath{\mu_h}}
\newcommand*{\testGlobCell}{\ensuremath{v_h}}
\newcommand*{\testGlobTrace}{\ensuremath{\nu_h}}
\newcommand*{\trialRecLoc}{\ensuremath{(\recLoc^k\mu_T, \mu_T)}}
\newcommand*{\testRecLoc}{\ensuremath{(\recLoc^k\nu_T, \nu_T)}}
\newcommand*{\trialRecGlob}{\ensuremath{(\recGlob^k\mu_h, \mu_h)}}
\newcommand*{\testRecGlob}{\ensuremath{(\recGlob^k\nu_h, \nu_h)}}
\newcommand*{\errGlob}{\ensuremath{(e_h, \epsilon_h)}}
\newcommand*{\errLoc}{\ensuremath{(e_T, \epsilon_T)}}
\newcommand*{\errLocCell}{\ensuremath{e_T}}
\newcommand*{\errLocTrace}{\ensuremath{\epsilon_T}}
\newcommand*{\errGlobCell}{\ensuremath{e_h}}
\newcommand*{\errGlobTrace}{\ensuremath{\epsilon_h}}

\newcommand*{\hor}{\ensuremath{p_T^{k+1}}}
\newcommand*{\horNoDeg}{\ensuremath{p_T}}
\newcommand*{\gradHor}{\ensuremath{\boldsymbol G^k_T }}
\newcommand*{\gradHorNoDeg}{\ensuremath{\boldsymbol G_T }}
\newcommand*{\interpolLocNoArg}{\ensuremath{(\projLLoc, \projLTraceLoc)}}
\newcommand*{\interpolLoc}[1]{\ensuremath{(\projLLoc #1, \projLTraceLoc #1)}}
\newcommand*{\interpolLocM}[1]{\ensuremath{(\projLLocM #1, \projLTraceLocM #1)}}
\newcommand*{\interpolGlobNoArg}{\ensuremath{(\projLGlob, \projLTraceGlob)}}
\newcommand*{\interpolGlob}[1]{\ensuremath{(\projLGlob #1, \projLTraceGlob #1)}}
\newcommand*{\interpolGlobM}[1]{\ensuremath{(\projLGlobM #1, \projLTraceGlobM #1)}}
\newcommand*{\recLoc}{\ensuremath{U_T}}
\newcommand*{\recGlob}{\ensuremath{U_h}}
\newcommand*{\vLoc}{\ensuremath{V_T^{k}}}
\newcommand*{\vGlob}{\ensuremath{V_h^{k}}}
\newcommand*{\recLocS}{\ensuremath{U_{T_S}}}

\newcommand*{\aNonCondLoc}{\ensuremath{\underline{a}_T^k}}
\newcommand*{\aNonCondLocM}{\ensuremath{\underline{a}_T^{\mDeg}}}
\newcommand*{\aNonCondGlob}{\ensuremath{\underline{a}_h^k}}
\newcommand*{\aNonCondGlobM}{\ensuremath{\underline{a}_h^{\mDeg}}}
\newcommand*{\aNonCondLocArg}[2]{\ensuremath{\underline{a}_T^k\big(#1, #2\big)}}
\newcommand*{\aNonCondGlobArg}[2]{\ensuremath{\underline{a}_h^k\big(#1, #2\big)}}
\newcommand*{\aNonCondLocArgM}[2]{\ensuremath{\underline{a}_T^{\mDeg}\big(#1, #2\big)}}
\newcommand*{\aNonCondGlobArgM}[2]{\ensuremath{\underline{a}_h^{\mDeg}\big(#1, #2\big)}}
\newcommand*{\aLoc}{\ensuremath{a_T^k}}
\newcommand*{\aGlob}{\ensuremath{a_h^k}}
\newcommand*{\aLocArg}[2]{\ensuremath{a_T^k(#1, #2)}}
\newcommand*{\aGlobArg}[2]{\ensuremath{a_h^k(#1, #2)}}
\newcommand*{\aGlobArgM}[2]{\ensuremath{a_h^{\mDeg}(#1, #2)}}

\newcommand*{\sNonCondLoc}{\ensuremath{\underline{s}_T^k}}
\newcommand*{\sNonCondLocM}{\ensuremath{\underline{s}_T^{\mDeg}}}

\newcommand*{\sNonCondLocArg}[2]{\ensuremath{\underline{s}_T^k}\big(#1, #2\big)}
\newcommand*{\sNonCondLocArgM}[2]{\ensuremath{\underline{s}_T^{\mDeg}}\big(#1, #2\big)}

\newcommand*{\sNormNonCondLoc}[1]{\ensuremath{|#1|_{\sNonCondLoc}}}
\newcommand*{\sNormNonCondLocM}[1]{\ensuremath{|#1|_{\sNonCondLocM}}}
\newcommand*{\aNormNonCondLoc}[1]{\ensuremath{\|#1\|_{\aNonCondLoc}}}
\newcommand*{\aNormNonCondLocNoDeg}[2]{\ensuremath{\|#2\|_{\underline{a}_T^{#1}}}}
\newcommand*{\aNormNonCondLocM}[1]{\ensuremath{\|#1\|_{\aNonCondLocM}}}
\newcommand*{\aNormNonCondGlob}[1]{\ensuremath{\|#1\|_{\aNonCondGlob}}}
\newcommand*{\aNormNonCondGlobM}[1]{\ensuremath{\|#1\|_{\aNonCondGlobM}}}
\newcommand*{\aNormLoc}[1]{\ensuremath{\|#1\|_{\aLoc}}}
\newcommand*{\aNormGlob}[1]{\ensuremath{\|#1\|_{\aGlob}}}
\newcommand*{\aNormGlobNoDeg}[2]{\ensuremath{\|#2\|_{a_h^{#1}}}}
\newcommand*{\aNormGlobM}[1]{\ensuremath{\|#1\|_{a_h^{\mDeg}}}}

\newcommand*{\nSmooth}{\ensuremath{N_{\mathrm{smooth}}}}

\newcommand*{\nLevel}{\ensuremath{N_{\mathrm{level}}}}
\newcommand*{\polyDegMax}{\ensuremath{k}}
\newcommand*{\smoother}{\ensuremath{R^{k_j}}}

\newcommand*{\projMGN}{\ensuremath{P^{\nDeg}}}

\newcommand*{\projMGJMinus}{\ensuremath{P^{k_{j-1}}}}
\newcommand*{\mDeg}{k_j}
\newcommand*{\nDeg}{k_{j-1}}
\newcommand*{\lDeg}{l}
\newcommand*{\tildeTrialGlobTrace}{\ensuremath{\widetilde \mu_h}}
\newcommand*{\E}{\ensuremath{E_h}}
\newcommand*{\FMuLoc}[1]{\ensuremath{\boldsymbol{F}_S^{#1}(\mu_h)}}
\newcommand*{\FMuGlob}[1]{\ensuremath{\boldsymbol{F}_h^{#1}(\mu_h)}}

\newcommand*{\tildePNTrialGlobTrace}{\ensuremath{\widetilde{\projMGN\mu_h}}}

\newcommand*{\FPNMuGlob}[1]{\ensuremath{\boldsymbol{F}_h^{#1}(\projMGN\mu_h)}}

\newcommand*{\errFac}{\ensuremath{\Big(\frac{h}{k}\Big)}}
\newcommand*{\errFacL}{\ensuremath{\Big(\frac{h}{\lDeg}\Big)}}
\newcommand*{\errFacN}{\ensuremath{\Big(\frac{h}{\nDeg}\Big)}}
\newcommand*{\errFacM}{\ensuremath{\Big(\frac{h}{\mDeg}\Big)}}
\newcommand*{\errFacLoc}{\ensuremath{\Big(\frac{h_T}{k}\Big)}}

\newcommand*{\errFacMLoc}{\ensuremath{\Big(\frac{h_T}{\mDeg}\Big)}}
\newcommand*{\errFacInvLoc}{\ensuremath{\Big(\frac{k}{h_T}\Big)}}

\begin{document}

\title{Optimal $hp$-error estimates and $p$-multigrid convergence for Hybrid High-Order discretizations of the Poisson equation}
\author{Daniele A. Di Pietro}
\author{Emil H\"{o}ssjer}
\affil{%
  IMAG, Universit\'e de Montpellier, CNRS, Montpellier 34090, France\\
  \email{daniele.di-pietro@umontpellier.fr},
  \email{emil.hossjer@umontpellier.fr}
}

\maketitle

\begin{abstract}
  This paper presents two new theoretical results for Hybrid High-Order (HHO) methods applied to elliptic problems.
  First, we establish $hp$-error estimates for the HHO discretization of the Poisson problem that achieve optimal approximation rates with respect to both the mesh size $h$ and the polynomial degree $k$.
  These results improve upon previous analyses of hybrid methods, whose convergence estimates were suboptimal in $k$.
  Second, building on these estimates, we develop and analyze a non-inherited $p$-multigrid solver for the statically condensed HHO system.
  We prove results that improve upon the corresponding theory available for other non-conforming methods and constitute, to the best of our knowledge, the first rigorous convergence analysis of a $p$-multigrid algorithm for HHO discretizations.
\end{abstract}



\section{Introduction}

In this paper, we prove two new results for Hybrid High-Order (HHO) discretizations of elliptic problems:
optimal error estimates in both the mesh size $h$ and the polynomial degree $k$, thereby improving the results of \cite{Aghili.Di-Pietro.ea:17},
and, building on these estimates, novel convergence results for $p$-multigrid iterations, yielding what appears to be the sharpest result currently available in the context of polytopal methods.
\smallskip

For the sake of simplicity, we focus on the following Poisson problem with homogeneous boundary conditions:
Given a convex polytope $\Omega\subset \IR^d$, $d\in \{2,3\}$, and a source term $f\in L^2(\Omega)$, find $u\in H_0^1(\Omega)$ such that
\begin{equation} \label{eq: continuous Poisson problem}
\left(\nabla u, \nabla v\right)_\Omega = (f, v)_\Omega \qquad \forall v\in H_0^{1}(\Omega),
\end{equation}
where $(\cdot,\cdot)_\Omega$ denotes the usual inner product of $L^2(\Omega)$.

The first main contribution of the paper is the proof of optimal $hp$-error estimates for the HHO approximation of problem \eqref{eq: continuous Poisson problem} in both the energy- and $L^2$-norms.
Here, ``optimal'' refers to the best polynomial approximation rates predicted by the $hp$-version of the Bramble--Hilbert lemma; see \cite{Babuska.Suri:87}.
Such optimal error estimates have previously been established, for example, for conforming finite elements \cite{Schwab:98} and discontinuous Galerkin (DG) methods \cite{Cangiani.Georgoulis.ea:14} on standard meshes.
Previous results for hybrid methods include \cite{Egger.Waluga:12}, where the Stokes problem was analyzed in the context of hybridizable DG methods,
and \cite{Aghili.Di-Pietro.ea:17, Dong.Ern:24}, which focused on HHO discretizations of elliptic problems. In all these cases, the resulting estimates were suboptimal with respect to the polynomial degree.

The estimate in \cite{Aghili.Di-Pietro.ea:17} is one order suboptimal in $p$ and the crucial element missing is a technical result which states that the $L^2$-orthogonal projector satisfies an optimal approximation property on traces. The paper \cite{Dong.Ern:24} makes this adjustment and they prove an estimate which is one half order suboptimal in $p$. They consider general boundary conditions however and previous $p$-analyses on DG methods \cite{Georgoulis.Hall.Melenk:10} have concluded that without additional regularity assumptions, $p$-error estimates which are a half order suboptimal in $p$ are in fact sharp. In the present paper we prove optimal $p$-error estimates for homogeneous Dirichlet boundary conditions. Our scheme is different from the one in \cite{Aghili.Di-Pietro.ea:17} in that we employ a full gradient reconstruction operator and we introduce a polynomial weight in the stabilization term. The use of a gradient reconstruction is motivated by the new trace estimate of the $L^2$-orthogonal projection operator. The paper \cite{Dong.Ern:24} also introduces a polynomial weight in the stabilization, however a different one than we do.

In the second part of the paper, building on the $hp$-error estimates discussed above, we analyze a $p$-multigrid algorithm for solving the HHO discretization of problem \eqref{eq: continuous Poisson problem}.
HHO methods follow a local design in which element unknowns are coupled only with the corresponding trace unknowns on their boundaries.
As a result, the former can be efficiently expressed locally in terms of the latter through a Schur complement construction, a process commonly referred to as \emph{static condensation}.
This yields a global \emph{condensed system} involving only face unknowns.
The assembly of the discrete problem as well as the static condensation procedure are both linear in the number of mesh elements and readily amenable to parallelization.
Consequently, for large-scale problems, the computational bottleneck lies in the solution of the condensed system.
A popular choice for this task is the use of multigrid methods.

Multigrid methods are iterative solvers that exploit discretization-specific information to achieve scalable, optimal-complexity performance \cite{Brenner.Scott:02}.
They consist of a \emph{smoother}, i.e., an iterative method that efficiently damps high-frequency error components,
and a \emph{correction step}, i.e., an approximate solution of the residual equation on a coarser level.
Common multigrid constructions employ $h$- and/or $p$-coarsening to define the coarse problems.
Standard frameworks for proving convergence are developed in \cite{Brenner.Scott:02,Brenner:04,Bramble.Pasciak.ea:91}.
In this paper, we rely on the framework of \cite{Bramble.Pasciak.ea:91}.

In the context of HHO methods, geometric $h$-multigrid strategies have been proposed and numerically investigated in \cite{Di-Pietro.Hulsemann.ea:21,Di-Pietro.Hulsemann.ea:21*1,Di-Pietro.Matalon.ea:23}; see also \cite{Cockburn.Dubois.ea:14,Roberts.Chan:16,Wildey.Muralikrishnan.ea:19} for related hybrid methods.
A crucial challenge in the design and analysis of $h$-multigrid methods for HHO discretizations stems from face coarsening, which typically generates coarse meshes with non-planar faces.
In this respect, $p$-multigrid provides a simpler and more natural alternative, in which coarsening is achieved by reducing the polynomial degree of both element and face unknowns.
In \cite{Botti.Di-Pietro:22}, a numerical study of the Stokes problem demonstrated robust convergence properties with respect to both $h$ and $p$; see also \cite{Franciolini.Fidkowski.ea:20} for an investigation in terms of CPU time and memory consumption.
Both works formulate the coarse problems by inheritance and use multigrid as a preconditioner for a Krylov subspace method.

From a theoretical standpoint, it was shown in \cite{Di-Pietro.Dong.ea:25} that a nested $h$-multigrid algorithm for HHO methods on conforming simplicial meshes converges uniformly with respect to $h$.
To the best of our knowledge, however, no theoretical analysis of the $p$-version has been carried out so far, and the present work aims precisely at filling this gap.
\smallskip

The analysis of $p$-multigrid algorithms for polytopal methods has previously been considered in \cite{Antonietti.Houston.ea:14,Antonietti.Mascotto.ea:17}.
Both contributions share several features with the present work: the target problem is \eqref{eq: continuous Poisson problem}; a non-inherited multigrid algorithm is considered, meaning that coarse-level problems are assembled independently rather than inherited from the fine level; and a Richardson iteration, or a smoother with sufficiently similar behavior, is employed.
The theoretical convergence result is also of the same form: Provided that the number of smoothing steps $\nSmooth$ is sufficiently large, then one multigrid iteration with the $W$-cycle contracts the error in some norm by a factor $\varrho(\nSmooth)$ which is independent of the mesh size $h$ but which has a polynomial deterioration in the polynomial degree $k$.
The paper \cite{Antonietti.Mascotto.ea:17}, in the context of $H^1$-conforming Virtual Element Methods, does not provide a precise expression for $\varrho(\nSmooth)$ and instead only a very coarse bound is given. The paper \cite{Antonietti.Houston.ea:14}, in the context of DG methods, however, serves as an interesting benchmark.
Their theorem and the result of the current paper are, respectively,
\begin{align}
  \label{eq: DG convergence result}
 \varrho^{\rm DG}(\nSmooth) &\le C^{\rm DG}\frac{k^2}{1 + \nSmooth},\\
  \label{eq: HHO convergence result}
  \varrho^{\rm HHO}(\nSmooth) &\le C^{\rm HHO}\frac{k^{\frac{7}{4}}}{1 + \nSmooth},
\end{align}
where the constants $C^{\rm DG}$ and $C^{\rm HHO}$ are independent of $h$ and $k$.

It is instructive to examine the origin of the $k$-dependencies in \eqref{eq: DG convergence result}--\eqref{eq: HHO convergence result}. They correspond to a product $\alpha_k\lambda_k$, where $\lambda_k$ denotes the largest eigenvalue of the system matrix and $\alpha_k$ quantifies how accurately the coarse problem approximates the fine one.
For DG methods, one has $\lambda_k^{\rm DG}\sim \frac{k^4}{h^2}$ and $\alpha_k^{\rm DG}\sim \frac{h^2}{k^2}$, which indeed yields a quadratic dependence in $k$.
For statically condensed HHO methods, one has $\lambda_k^{\rm HHO}\sim \frac{k^2}{h}$ and $\alpha_k^{\rm HHO}\sim \frac{h}{k^{\frac14}}$ (in passing, the latter estimate may potentially be improved).
Heuristically, the improved $k$-dependence in \eqref{eq: HHO convergence result} over \eqref{eq: DG convergence result} can be attributed to the presence of trace unknowns.
 Indeed, the $k$-dependence of $\lambda_k$ typically mirrors its $h$-dependence, and scaling arguments indicate that the latter is weaker for trace unknowns than for element unknowns.
Without static condensation, HHO methods would therefore not exhibit an improved $k$-dependence over DG methods in the multigrid convergence rate.
\smallskip

The remainder of the paper is organized as follows.
In Section \ref{section: setting}, we introduce the mesh together with the notation, assumptions, and inequalities used throughout the paper.
In Section \ref{section: hho schemes}, we define the HHO scheme, review static condensation, and state our main $hp$-error estimates.
In Section \ref{section: p-multigrid}, we introduce the multigrid algorithm, state the convergence theorem, and present numerical results.
Finally, Sections \ref{section: hp error estimates (analysis)} and \ref{section: multigrid analysis} are devoted to the proofs of the $hp$-error estimates and the multigrid convergence theorem, respectively.


\section{Setting}\label{section: setting}

\subsection{Mesh}\label{section: mesh}

For any measurable set $X\subseteq \IR^d$, we let $h_X$ denote its diameter.
We denote by $(\mesh, \faces)_h$ a polytopal regular mesh sequence in the sense of \cite[Definition 1.9]{Di-Pietro.Droniou:20}, with $\mesh$ collecting the mesh elements and $\faces$ the mesh faces.
In particular, there exists a shape regular sequence of simplicial submeshes $(\submesh, \subfaces)_h$ and a constant $c>0$ independent of $h$ such that:
$$
\text{For any $S\in\submesh$, there exists a unique $T\in\mesh$ such that $S\subseteq T$ and $h_T \le c h_S$.}
$$
The mesh size of $\mesh$ is $h = \max_{T\in\mesh} h_T$ and we shall assume quasi-uniformity, i.e., the existence of $C>0$ independent of $h$ such that
\begin{equation}
	\label{eq: mesh quasi uniformity assumption}
	h \le C h_T  \qquad \forall T\in\mesh.
\end{equation}
We furthermore set, for all $T\in\mesh$, $\facesLoc \coloneqq \{F \in \faces : F\subset \partial T\}$.
Finally, we define the mesh skeleton $\facesUnion \coloneqq \bigcup_{F \in\faces}\overline{F}$ and,
for each $T \in \mesh$, we denote by $\normalT$ the unit normal field on $\partial T$ pointing out of $T$.

In practice, to obtain the elliptic regularity condition of Section \ref{section: elliptic regularity} below, we must also assume that all the mesh elements $T\in\mesh$ are convex.

\subsection{Inner products and norms}

With $X$ as before, for any $q\ge 0$, we denote by $H^q(X)$ the usual Hilbert space of $q$ times weakly differentiable functions in $L^2(X) = H^0(X)$.
We respectively denote by $(\cdot,\cdot)_{q,X}$ and $\|\cdot\|_{q,X}$ the usual inner product and norm of $H^q(X)$. For $q=0$, we simply write $(\cdot,\cdot)_{X}$ and $\|\cdot\|_{X}$. A particularly relevant role will be played by the inner products $(\cdot,\cdot)_{\facesUnion}$ on the mesh skeleton and $(\cdot,\cdot)_{\partial T}$ on the boundary of a generic element $T\in\mesh$ together with the corresponding norms $\|\cdot\|_{\facesUnion}$ and $\|\cdot\|_{\partial T}$.

\subsection{Local polynomial spaces and projectors}

Let $k\in\IN$. For any $X\in\mesh\cup\faces$, we denote by $\IP^k(X)$ the space spanned by the restrictions to $X$ of polynomials of total degree $\le k$ in the space variables.
At the global level, we define the following broken spaces on the mesh and its skeleton:
\begin{align*}
  \IP^k(\mesh) &\coloneqq \{v \in L^2(\Omega) : v|_T\in \IP^k(T) \text{ for any } T\in\mesh\},\\
  \IP^k(\faces) &\coloneqq \{v \in L^2(\facesUnion) : v|_F\in \IP^k(F) \text{ for any } F\in\faces\},
\end{align*}
and, locally for any $T\in\mesh$, we let
\[
\IP^k(\facesLoc) \coloneqq \{v \in L^2(\partial T) : v|_F\in \IP^k(F) \text{ for any } F\in\facesLoc\}.
\]
For any $X \in \mesh \cup \faces$, the $L^2$-orthogonal projector $\Pi^k_X : L^2(X)\rightarrow \IP^k(X)$ is defined such that, for any $v\in L^2(X)$,
\begin{align*}
  (\Pi^k_X v, w )_X &= (v, w)_X \qquad \forall w\in \IP^k(X).
\end{align*}
The projector $\boldsymbol \Pi^k_X : L^2(X)^d\rightarrow \IP^k(X)^d$ acting on vector-valued fields is obtained applying $\Pi^k_X$ component-wise.
The global skeletal $L^2$-projector $\projLTraceGlob : L^2\left(\facesUnion \right)\rightarrow \IP^k(\faces)$ is defined such that, for any $v\in L^2\left(\facesUnion \right)$,
\begin{align*}
  (\projLTraceGlob v, w)_{\facesUnion} &= (v, w)_\facesUnion \qquad \forall w\in \IP^k(\faces).
\end{align*}
Its local counterpart on an element $T \in \mesh$ is $\projLTraceLoc : L^2\left(\partial T \right)\rightarrow \IP^k(\facesLoc)$ such that, for any $v\in L^2\left(\partial T \right)$,
\begin{align*}
  (\projLTraceLoc v, w )_{\partial T} &= (v, w)_{\partial T} \qquad \forall w\in \IP^k(\facesLoc).
\end{align*}
These projectors are obviously weakly contractive in their respective norms.
In particular, we have
\begin{equation}
  \label{eq: stability of L2-projection on local traces}
  \|\projLTraceLoc u\|_{\partial T} \le \|u\|_{\partial T} \qquad \forall u\in L^2(\partial T).
\end{equation}

Finally, for any $T \in \mesh$, the elliptic projector $\varpi^k_{T} : H^1(T)\rightarrow \IP^k(T)$ is defined such that, for any $v\in H^1(T)$,
\begin{align*}
  (\nabla\varpi^k_{T} v, \nabla w)_{T} &= (\nabla v, \nabla w)_{T} \qquad \forall w\in \IP^k(T),
  \\
  \int_{T} \varpi^k_{T} v\dx &= \int_{T} v\dx.
\end{align*}

\subsection{Local inequalities}\label{sec: setting: inequalities}

From this point on, we will write $a \lesssim b$ to signify that there exists a constant $C> 0$ independent of $h$ and any involved polynomial degree such that $a\le Cb$. For local inequalities on a mesh element $T\in\mesh$, the constant $C$ is additionally independent of $T$.

From \cite[Lemma 1.41]{Di-Pietro.Ern:12}, we have the multiplicative trace inequality
\begin{equation}
  \label{eq: multiplicative trace inequality}
  \|v\|^2_{\partial T} \lesssim \|v\|_T\|\nabla v\|_T + h_T^{-1}\|v\|_T^2 \qquad \forall v\in H^1(T).
\end{equation}
Let $k\ge 1$ be a polynomial degree and let $m\in\IN$ satisfy $1\le m \le k+1$.
From \cite{Aghili.Di-Pietro.ea:17}, we state the following two results on $hp$-analysis. The first is the discrete trace inequality
\begin{align}
  \label{eq: discrete inverse trace inequality}
  \|v\|_{\partial T} &\lesssim \frac{k}{h_T^{\frac{1}{2}}} \|v\|_T \qquad \forall v\in \IP^k(T).
\end{align}
The second is the fundamental approximation result in $hp$-analysis:
For any $T\in\mesh$ and any $v\in H^m(T)$, there exists $P_T^kv\in\IP^k(T)$ such that, for all $0\le q\le m$, there holds
\begin{equation}
  \label{eq: fundamental approximation result}
  \|v - P^k_Tv\|_{q,T} \lesssim \Big(\frac{h_T}{k}\Big)^{m-q}\|v\|_{m,T}.
\end{equation}

Successively applying \eqref{eq: multiplicative trace inequality} and \eqref{eq: fundamental approximation result}, we infer also, for all $0\le q\le m - 1$,
\begin{equation}
  \label{eq: fundamental approximation result on trace}
  \|v - P^k_Tv\|_{q,\partial T} \lesssim \Big(\frac{h_T}{k}\Big)^{m-q - \frac{1}{2}}\|v\|_{m,T}.
\end{equation}
The approximation properties in $h_T$ and $k$ provided in equations \eqref{eq: fundamental approximation result} and \eqref{eq: fundamental approximation result on trace} characterize the \emph{optimal} error decay rates in $hp$-analysis.

\begin{lem}[Approximation properties of the $L^2$-orthogonal and elliptic projectors] \label{lemma: Approximation of L2 and elliptic projectors}
  For any $T\in\mesh$ and $v\in H^m(T)$, we have
  \begin{align}
    \label{eq: approximation L2-projection on cell}
    \|v - \projLLoc v\|_{T} &\lesssim \Big(\frac{h_T}{k}\Big)^{m}\|v\|_{m,T},
    \\
    \label{eq: approximation elliptic-projection in first Sobolev norm on cell}
    \|\nabla v - \nabla\projEllipticLoc v\|_T &\lesssim \Big(\frac{h_T}{k}\Big)^{m-1}\|v\|_{m,T}.
  \end{align}
\end{lem}
\begin{proof}
  The relations are immediate from \eqref{eq: fundamental approximation result} as each of the projections gives a minimizing polynomial with respect to the left hand-side norm.
\end{proof}

We shall furthermore assume the following $hp$-trace approximation results:
For any $T\in\mesh$ and any $v\in H^m(T)$,
\begin{equation}
  \label{eq: approximation L2-projection on trace}
  \|v - \projLLoc v\|_{\partial T} \lesssim \Big(\frac{h_T}{k}\Big)^{m - \frac{1}{2}}\|v\|_{m,T}
\end{equation}
and
\begin{equation}
  \label{eq: approximation L2-projection on trace FIRST ORDER}
  \|v - \projLLoc v\|_{\partial T} \lesssim \errFacLoc^{ \frac{1}{2}}\|\nabla v\|_{T}.
\end{equation}

\begin{rem}[$hp$-approximation results]
  The $p$-version of \eqref{eq: approximation L2-projection on trace FIRST ORDER} is proved for simplices in \cite{Chernov:12}.
  The $hp$-version can be inferred from a standard scaling argument.
  In \cite{Houston.Schwab.ea:02}, the authors prove \eqref{eq: approximation L2-projection on trace} in arbitrary space dimension when using tensor-product polynomial spaces, with $H^m$-norm replaced by the $H^m$-seminorm in the right-hand side.
  
  Very recent results communicated to us by Simon Lemaire yield improvements of the above classical estimates.
  In particular:
  \begin{itemize}
  \item all the occurrences of the norm $\| v \|_{m,T}$ in this section can be replaced by the seminorm $| v |_{m,T}$.
    This results in similar changes to the estimates \eqref{eq: theorem non-condensed energy error estimate} and \eqref{eq: theorem non-condensed L2 error estimate} below;
  \item the trace approximation property \eqref{eq: approximation L2-projection on trace} (with $\| v \|_{m,T}$ replaced by $| v |_{m,T}$) has been proved on star-shaped elements.
    From this relation, \eqref{eq: approximation L2-projection on trace FIRST ORDER} follows immediately.
  \end{itemize}
  Notice that the best approximation polynomial $P_T^k v$ in the improved versions of \eqref{eq: fundamental approximation result} and \eqref{eq: fundamental approximation result on trace} also depends on $m$ (which is not a concern, since $m$ is a fixed parameter in the error estimates of Section \ref{section: hho main results}).
  A paper detailing the above improvements is in the process of being finalized, and we will add a precise reference in the final version of the present work.
\end{rem}

\subsection{Elliptic regularity assumption} \label{section: elliptic regularity}

Let $D$ denote either $\Omega$ or a mesh element $T\in\mesh$. Let $g\in L^2(D)$ and let $w\in H_0^1(D)$ solve
\begin{equation*}
  (\nabla w, \nabla v)_D = (g, v)_D \qquad \forall v\in H_0^{1}(D).
\end{equation*}
We shall assume that then $w\in H^2(D)$ and
\begin{equation}
  \label{eq: elliptic regularity}
  \|w\|_{2,D} \lesssim \|g\|_D.
\end{equation}
In the case $d=2$ and $T\in\mesh$ convex, these assumptions hold as a result of \cite[Theorem 4.3.1.4]{Grisvard:11}, together with the inequalities
\begin{align*}
  \|w\|_D &\lesssim h_D^2 \|\Delta w\|_D\\
  \|\nabla w\|_D &\lesssim h_D \|\Delta w\|_D,
\end{align*}
which are derived from the Poincaré inequality $\|w\|_D \le \pi^{-1}  h_D\|\nabla w\|_D$ (see \cite{Payne.Weinberger:60,Bebendorf:03}) and an integration by parts of $(\nabla w, \nabla w)_D$.
These inequalities together with the fact that, for all choices of $D$, $h_D \le h_\Omega \lesssim 1$ show that the hidden constant in \eqref{eq: elliptic regularity} can indeed be taken independent of $T\in\mesh$.

\begin{rem}[Elliptic regularity]
  Assumption \eqref{eq: elliptic regularity} will be needed for the condensed error analysis of Section \ref{section: condensed error analysis} and for the multigrid analysis of Section \ref{section: multigrid analysis non-inherited W-cycle}.
  One could carry out these analyses with a weaker regularity assumption as well, which would weaken the corresponding results and introduce a dependence on the regularity parameter, but for simplicity we shall assume the full $H^2$-regularity.
\end{rem}


\section{Non-condensed and condensed HHO schemes} \label{section: hho schemes}

We formulate the HHO scheme as well as its condensed version, and state $hp$-error estimates for both.
Throughout this section, the integer $k\ge1$ denotes the polynomial degree of the scheme. We exclude the case $k=0$ as its analysis is different from the case $k>0$ and its $hp$-analysis is equivalent to just $h$-analysis, which is already optimally provided in for example \cite{Di-Pietro.Droniou:20}.

\subsection{Non-condensed HHO scheme}

The discrete space for the HHO formulation at the global level is $\hhoSpaceGlob$, where
\begin{equation}\label{eq:global hho spaces}
  \hhoSpaceCell_h^k \coloneqq \IP^k(\mesh),\qquad
  \hhoSpaceTrace_h^k \coloneqq \left\{\mu_h\in \IP^k(\faces) :
  \text{
    $\mu_h|_F = 0$ for all $F \in \faces$  such that $F \subset \partial \Omega$
  }
  \right\}.
\end{equation}
Locally, for any $T\in\mesh$, the discrete space is $\hhoSpaceLoc$, where
\[
\hhoSpaceCell_T^k \coloneqq \IP^k(T),\qquad
\hhoSpaceTrace_T^k \coloneqq \IP^k(\facesLoc).
\]
For any $\trialGlob\in\hhoSpaceGlob$, we denote by $\trialLoc\in\hhoSpaceLoc$ the restriction to the local components on $T$, for any $T\in\mesh$.
Let $\trialGlob, \testGlob\in\hhoSpaceGlob$. We define the following operations which turn $\hhoSpaceGlob$ into a vector space,
\[
\begin{alignedat}{2}
  \trialGlob+ \testGlob &= \left(\trialGlobCell + \testGlobCell, \trialGlobTrace + \testGlobTrace\right),
  \\
  c \trialGlob &= \left(c\trialGlobCell, c\trialGlobTrace\right) &\qquad& \forall c\in\IR.
\end{alignedat}
\]
We put a similar vector space structure on the local space $\hhoSpaceLoc$. The interpolation of regular enough functions is obtained collecting $L^2$-orthogonal projections on elements and faces:
\[
\begin{aligned}
  \interpolGlobNoArg : H^1(\Omega)&\rightarrow \hhoSpaceGlob,
  \\
  \interpolLocNoArg : H^1(T)&\rightarrow \hhoSpaceLoc.
\end{aligned}
\]
The HHO scheme is based on local operators.  For any $T\in\mesh$, we define the \emph{gradient reconstruction operator} $\gradHor : \hhoSpaceLoc\rightarrow \IP^k(T)^d$ such that, for all $\trialLoc\in\hhoSpaceLoc$,
\begin{equation}
  \label{eq: gradHor definition}
  (\gradHor\trialLoc, \genericVec )_T
  = (\nabla \trialLocCell, \genericVec )_T
  + (\trialLocTrace - \trialLocCell, \genericVec\cdot \normalT)_{\partial T}
  \qquad \forall \genericVec\in \IP^{k}(T)^d.
\end{equation}
The \emph{potential reconstruction operator} $\hor : \hhoSpaceLoc\rightarrow \IP^{k+1}(T)$ is such that, for all $\trialLoc\in\hhoSpaceLoc$,
\[
\begin{alignedat}{2}
  (\nabla\hor\trialLoc, \nabla w )_T &= (\gradHor\trialLoc, \nabla w )_T &\qquad& \forall w\in \IP^{k+1}(T),\\
  \int_T\hor\trialLoc\dx &= \int_T\trialLocCell\dx.
\end{alignedat}
\]
In particular, $\nabla\hor\trialLoc$ is the $L^2$-projection of $\gradHor\trialLoc$ onto $\nabla \IP^{k+1}(T)$ and so
\begin{equation}
  \label{eq: L2-norm hor <= L2-norm gradHor}
  \|\nabla\hor\trialLoc\|_T \le \|\gradHor\trialLoc\|_T.
\end{equation}
We have the following fundamental properties \cite{Di-Pietro.Droniou:20}: For all $w\in H^1(T)$,
\begin{align}
  \label{eq: hor commutation with interpolation}
  \hor\interpolLoc{w} &= \projEllipticLocPlus w,\\
  \label{eq: gradHor commutation with interpolation}
  \gradHor\interpolLoc{w} &= \projLLocVec\nabla w.
\end{align}

We next define the bilinear form
\begin{align*}
  \aNonCondGlob &: \left[\hhoSpaceGlob\right]^2 \rightarrow \IR
\end{align*}
such that, for all $\trialGlob, \testGlob\in\hhoSpaceGlob$, 
\begin{equation}\label{eq:aNonCondGlob}
  \begin{aligned}
    \aNonCondGlob\left(\trialGlob, \testGlob\right) & \coloneqq \sumT \aNonCondLoc\left(\trialLoc, \testLoc\right),
    \\
    \aNonCondLoc
    \left(\trialLoc, \testLoc\right) & \coloneqq \big(\gradHor\trialLoc, \gradHor\testLoc\big)_T
    + \sNonCondLoc\left(\trialLoc, \testLoc\right),
  \end{aligned}
\end{equation}
where $\sNonCondLoc$ is a symmetric and non-negative stabilization term.
  We define the following global norm :
  \begin{equation}\label{eq: aNormNonCondGlob}
    \aNormNonCondGlob{\testGlob} \coloneqq \sqrt{\aNonCondGlob\left(\testGlob, \testGlob\right)},
\end{equation}
as well as the local semi-norms
\begin{equation}\label{eq:non condensed norms}
  \aNormNonCondLoc{\testLoc} \coloneqq \sqrt{\aNonCondLoc\left(\testLoc, \testLoc\right)},\qquad
  \sNormNonCondLoc{\testLoc} \coloneqq \sqrt{\sNonCondLoc\left(\testLoc, \testLoc\right)}.
\end{equation}

We want the stabilization to satisfy the following two properties:
\begin{alignat}{2}
  \label{eq: stabilization: trace - cell dof bound}\tag{S1}
  \|\testLocTrace - \testLocCell\|_{\partial T} &\lesssim \Big(\frac{h_T}{k}\Big)^{\frac{1}{2}} \aNormNonCondLoc{\testLoc} &\qquad& \forall \testLoc\in\hhoSpaceLoc,\\
  \label{eq: stabilization consistency}\tag{S2}
  \sNormNonCondLoc{\interpolLoc w} &\lesssim \Big(\frac{h_T}{k}\Big)^{m-1}\|w\|_{m,T} &\qquad& \forall w\in H^m(T) \text{ with } 1\le m \le k+1.
\end{alignat}
Property \eqref{eq: stabilization: trace - cell dof bound} implies that $\aNonCondGlob$ is coercive with the correct scaling in both $h_T$ and $k$. Property \eqref{eq: stabilization consistency} expresses the consistency of the stabilization term.
In practice, this consistency property is achieved by requiring that the stabilization vanish whenever one of its arguments is the interpolant of a polynomial of degree $\le k+1$.
For this to hold, $\sNonCondLoc$ must depend on its argument through the difference operators
$\big(\delta_T^k, \delta_{\partial T}^k\big):\hhoSpaceLoc\rightarrow \hhoSpaceLoc$ which, for any $\testLoc\in\hhoSpaceLoc$, are defined by
\begin{equation}\label{eq:difference operators}
  \big(\delta_T^k\testLoc, \delta_{\partial T}^k\testLoc\big) \coloneqq \testLoc - \interpolLoc{\hor \testLoc}.
\end{equation}
We shall employ the following choice of stabilization:
\begin{equation}
  \label{eq: standard HHO stabilization}
  \sNonCondLoc\left(\trialLoc, \testLoc \right) = \frac{k}{h_T}\big(\delta_{\partial T}^k\trialLoc - \delta_T^k\trialLoc, \delta_{\partial T}^k\testLoc - \delta_T^k\testLoc \big)_{\partial T}.
\end{equation}
This is the standard choice for HHO, with the exception that the polynomial weight $k$ is new. This weight guarantees that the optimal $k$-dependency in \eqref{eq: stabilization: trace - cell dof bound} and \eqref{eq: stabilization consistency} is satisfied; see Lemma \ref{lem: stabilization consistency} below.
We can then state the HHO discretization of \eqref{eq: continuous Poisson problem}: Find $\trialGlob\in\hhoSpaceGlob$ such that
\begin{equation}
  \label{eq: global problem discretized non-condensed}
  \aNonCondGlob\left(\trialGlob, \testGlob\right) = (f, \testGlobCell)_\Omega \qquad \forall \testGlob\in\hhoSpaceGlob.
\end{equation}

\begin{rem}[VEM-inspired stabilization]
  Another natural choice of stabilization would be one inspired from the virtual element method, but with new polynomial weights, given by
  \begin{equation*}
    \sNonCondLoc\left(\trialLoc, \testLoc \right) = \frac{k}{h_T}\big(\delta_{\partial T}^k\trialLoc, 	\delta_{\partial T}^k\testLoc\big)_{\partial T} + \Big(\frac{k}{h_T}\Big)^2\big(\delta_{T}^k\trialLoc, \delta_{T}^k\testLoc\big)_{T}.
  \end{equation*}
  However, so far we have been short a factor of $k^{\frac{1}{4}}$ in our attempts to show \eqref{eq: stabilization: trace - cell dof bound} with this choice for $\underline{s}_T^k$.
\end{rem}

\begin{rem}[Potential-based consistency term]
  A classical variant of \eqref{eq: global problem discretized non-condensed} is obtained replacing $\gradHor$ with $\nabla\hor$ in the consistency term.
  However, to obtain optimal $hp$-error estimates for this scheme, one  would need optimal $hp$-approximation of $\nabla\projEllipticLoc$ on traces, corresponding to \ref{eq: approximation L2-projection on trace}.
  Proving these approximation properties is an open question on which we are currently working.
\end{rem}

\subsection{Condensed HHO scheme} \label{sec: condensed HHO scheme}
The condensed formulation presented in this section was derived in \cite{Cockburn.Di-Pietro.ea:16}.
For all $T\in\mesh$, we apply static condensation to the element component in $\hhoSpaceCell_T^k$, which means that this component is solved for in terms of the face component in $\hhoSpaceTrace_T^k$.
Specifically, we define the local operator $\recLoc^k : \hhoSpaceTrace_T^k \rightarrow \hhoSpaceCell_T^k$, which, for any $\trialLocTrace\in\hhoSpaceTrace_T^k$, is obtained solving
\begin{equation}
  \label{eq: U_T^k definition}
  \aNonCondLocArg{(\recLoc^k\trialLocTrace, 0)}{(\testLocCell, 0)} = -\aNonCondLocArg{(0, \trialLocTrace)}{(\testLocCell, 0)} \qquad \forall \testLocCell\in\hhoSpaceCell_T^k,
\end{equation}
and the local operator $\vLoc : L^2(T) \rightarrow \hhoSpaceCell_T^k$, which, for any $g\in L^2(T)$, is given by
\begin{equation}
  \label{eq:vLoc}
  \aNonCondLocArg{(\vLoc g, 0)}{(\testLocCell, 0)} = (g, \testLocCell)_T  \qquad \forall \testLocCell\in\hhoSpaceCell_T^k.
\end{equation}
These operators are well-defined because $\hhoSpaceCell_T^k\times \{0\}$ is the discrete space in the (non-condensed) HHO scheme for a Poisson problem on $T$ with homogeneous Dirichlet boundary condition, so $\aNonCondLoc$ is coercive on this set.
We also define the global versions $\recGlob^k : \hhoSpaceTrace_h^k \rightarrow \hhoSpaceCell_h^k$ and $\vGlob: L^2(\Omega) \rightarrow \hhoSpaceCell_h^k$ such that, for all $\trialGlobTrace \in \hhoSpaceTrace_h^k$ and all $f \in L^2(\Omega)$ satisfy, for all $T\in\mesh$,
\[
\big(\recGlob^k\trialGlobTrace\big)|_T \coloneqq \recLoc^k\trialLocTrace,\qquad
\big(\vGlob f\big)|_T \coloneqq \vLoc f|_T.
\]

The condensed version of the bilinear form from the last section is $\aGlob: \hhoSpaceTrace_h^k\times \hhoSpaceTrace_h^k \rightarrow \IR$ such that, for all $\trialGlobTrace, \testGlobTrace\in\hhoSpaceTrace_h^k$,
\begin{equation}\label{eq: aGlob}
\begin{aligned}
  \aGlobArg{\trialGlobTrace}{\testGlobTrace} &\coloneqq
  \aNonCondGlobArg{\trialRecGlob}{\testRecGlob}
  = \sumT \aLocArg{\trialLocTrace}{\testLocTrace},
  \\
  \aLocArg{\trialLocTrace}{\testLocTrace} &\coloneqq \aNonCondLocArg{\trialRecLoc}{\testRecLoc},
\end{aligned}
\end{equation}
and the condensed energy norm $\aNormGlob{\cdot}$ is given by
\begin{equation}\label{eq:aNormGlob}
  \aNormGlob{\testGlobTrace} \coloneqq \aNormNonCondGlob{(\recGlob^k \testGlobTrace, \testGlobTrace)}
  \qquad \forall \testGlobTrace\in\hhoSpaceTrace_h^k.
\end{equation}

The condensed scheme reads: Find $\trialGlob\in\hhoSpaceGlob$ such that
\begin{subequations}\label{eq: condensed problem}
  \begin{align}
    \label{eq: global problem discretized condensed}
    \aGlobArg{\trialGlobTrace}{\testGlobTrace} &= (f, \recGlob^k\testGlobTrace)_\Omega \qquad \forall \testGlobTrace\in\hhoSpaceTrace_h^k,\\
    \label{eq: condensed and non-condendensed solutions relation}
    \trialGlobCell &= \recGlob^k\trialGlobTrace + \vGlob f.
  \end{align}
\end{subequations}
Problem \eqref{eq: condensed problem} is equivalent to \eqref{eq: global problem discretized non-condensed}.

\subsection{Main results}
\label{section: hho main results}

  In the following results, the regularity requirements on the exact solution are expressed in terms of the broken Sobolev spaces
  \[
  H^m(\mesh) \coloneqq \left\{
  v \in L^2(\Omega) : \text{
    $v|_T \in H^m(T)$ for all $T \in \mesh$
  }
  \right\},
  \]
  endowed with the natural norm
  \[
  \| v \|_{m,h} \coloneqq \sqrt{
    \sum_{T \in \mesh}  \| v \|_{m,T}^2
   }.
  \]

\begin{thm}[Non-condensed energy error estimate]
  \label{theorem: non-condensed energy error estimate}
  Let $\trialGlob\in\hhoSpaceGlob$ denote the HHO solution \eqref{eq: global problem discretized non-condensed}, and let $u\in H_0^1(\Omega)$ denote the continuous solution \eqref{eq: continuous Poisson problem}.
  For any integer $m$ such that $2 \le m \le k+1$, assuming $u\in H^{m}(\mesh)$, we have
  \begin{equation}
    \label{eq: theorem non-condensed energy error estimate}
    \aNormNonCondGlob{\trialGlob - \interpolGlob{u}} \lesssim\errFac^{m-1}\|u\|_{m,h}.
  \end{equation}
\end{thm}

\begin{thm}[$L^2$-error estimate]
  \label{theorem: non-condensed L2-error estimate}
  In the setting of Theorem \ref{theorem: non-condensed energy error estimate}, and assuming additionally $f\in H^{m-2}(\mesh)$, we have
  \begin{equation}
    \label{eq: theorem non-condensed L2 error estimate}
    \|\trialGlobCell - \projLGlob u\|_\Omega \lesssim\errFac^{m}\Big(\|u\|_{m,h} + \|f\|_{m-2,h}\Big).
  \end{equation}
\end{thm}

\begin{thm}[Condensed energy error estimate to first order]
  \label{theorem: condensed energy error estimate}
  Let $\trialGlobTrace\in\hhoSpaceTrace_h^k$ denote the trace component of the HHO solution \eqref{eq: global problem discretized condensed} and let $u\in H_0^1(\Omega)$ denote the continuous solution \eqref{eq: continuous Poisson problem}. We have
  \begin{equation}
    \label{eq: theorem condensed energy error estimate}
    \aNormGlob{\trialGlobTrace - \projLTraceGlob u} \lesssim \frac{h}{k}\|f\|_{\Omega}.
  \end{equation}
\end{thm}

\begin{thm}[Trace $L^2$-error estimate to first order]
  \label{theorem: condensed L2-error estimate}
  In the setting of Theorem \ref{theorem: condensed energy error estimate}, we have
  \begin{equation}
    \label{eq: theorem condensed L2 error estimate}
    \|\trialGlobTrace - \projLTraceGlob u\|_\facesUnion \lesssim k^{\frac14}\errFac^{\frac{3}{2}}\|f\|_{\Omega}.
  \end{equation}
\end{thm}


\section{$p$-multigrid} \label{section: p-multigrid}

We consider a multigrid algorithm in the polynomial degree to solve the condensed problem \eqref{eq: global problem discretized condensed}.

\subsection{Algorithm}
\label{section: multigrid algorithms}

We let $1 = k_1 < \cdots < k_{\nLevel} = \polyDegMax$ denote a sequence of polynomial degrees, with each degree corresponding to a level in the algorithm. With $a_h^{k_j}$ denoting the condensed bilinear form in degree $k_j$ from Section \ref{sec: condensed HHO scheme}, we define the associated operator $A_h^{k_j}:\hhoSpaceTrace_h^{k_j}\rightarrow\hhoSpaceTrace_h^{k_j}$ such that, for all $\trialGlobTrace\in \hhoSpaceTrace_h^{k_j}$,
\begin{align}
  \label{eq: operator corresponding to general bilinear form level j}
  (A_h^{k_j} \trialGlobTrace, \testGlobTrace)_\facesUnion =  a_h^{k_j}(\trialGlobTrace, \testGlobTrace) \qquad \forall \testGlobTrace\in \hhoSpaceTrace_h^{k_j}.
\end{align}
We also denote by $\smoother : \hhoSpaceTrace_h^{k_j}\rightarrow\hhoSpaceTrace_h^{k_j}$ a \emph{smoother}, which we will assume to satisfy the minimal assumptions presented in \cite{Bramble.Pasciak.ea:91}. Notably, the smoother $\smoother$ depends on $A_h^{k_j}$. All standard choices of smoother such as the Richardson iteration and differently ordered Gauss--Seidel and Jacobi iterations are known to fulfil the assumptions \cite{Bramble.Pasciak:92}.
On each level $j=1,\ldots,\nLevel$, we want to solve a problem of the form: For some $\phi_h\in\hhoSpaceTrace_h^{k_j}$, find
$\trialGlobTrace\in\hhoSpaceTrace_h^{k_j}$ such that
\begin{equation}
  \label{eq: multigrid problem on level j}
  A_h^{k_j}\trialGlobTrace = \phi_h.
\end{equation}
The right hand-side of \eqref{eq: global problem discretized condensed} can be recast into this form by defining $\phi_h\in\hhoSpaceTrace^{k}_h$ as the unique solution to
\begin{equation}
  \label{eq: definition varphi_h}
  \left(\phi_h, \testGlobTrace \right)_{\facesUnion} = (f, \recGlob^{k}\testGlobTrace)_\Omega \qquad \forall \testGlobTrace\in\hhoSpaceTrace_h^{k}.
\end{equation}
We define a multigrid operator $B_h^{k_j} : \hhoSpaceTrace_h^{k_j}\rightarrow\hhoSpaceTrace_h^{k_j}$ which exactly inverts $A_h^{k_j}$ on the coarsest level and which otherwise is defined by induction.

\begin{algo}[Non-inherited $W$-cycle]
  \label{algorithm: general multigrid cycle}
  Let $\nSmooth$ be a positive integer.
  We define the action of $B_h^{k_j}$ on $\phi_h \in \hhoSpaceTrace_h^{k_j}$ as follows:
  \begin{enumerate}
  \item If $j=1$, set $B_h^{k_j}\phi_h = (A_h^{k_j})^{-1}\phi_h$. Otherwise, let $\lambda_h^{(0)} = 0\in \hhoSpaceTrace_h^{k_j}$.

  \item \emph{Presmooth}. For $i=1,\ldots, \nSmooth$,
    $$
    \lambda_h^{(i)} = \smoother\big(\phi_h - A_h^{k_j}\lambda_h^{(i-1)}\big) 
    $$
  \item \emph{Coarse correction}. For $i=1,2$,
    $$
    \lambda_h^{(\nSmooth+i)} = \lambda_h^{(\nSmooth + i - 1)} + B_h^{k_{j-1}}\projLTraceGlobArbitraryDeg^{k_{j-1}}\big(\phi_h - A_h^{k_j}\lambda_h^{(\nSmooth + i - 1)} \big) 
    $$
  \item \emph{Postsmooth}. For $i=1,\ldots, \nSmooth$,
    $$
    \lambda_h^{(\nSmooth + 2 + i)} = \smoother\big(\phi_h - A_h^{k_j}\lambda_h^{(\nSmooth + 2 + i - 1)}\big)
    $$

  \item Set $B_h^{k_j}\phi_h = \lambda_h^{(2\nSmooth + 2)}$.
  \end{enumerate}
\end{algo}

One multigrid iteration to solve \eqref{eq: multigrid problem on level j} at the finest level $j = \nLevel$ changes a current estimate $\trialGlobTrace^{(i)}\in \hhoSpaceTrace_h^{\polyDegMax}$ to
\begin{equation}
  \label{eq: non-inherited multigrid iterative relation}
  \trialGlobTrace^{(i + 1)} = \big(I - B_h^{\polyDegMax}A_h^{\polyDegMax}\big)\trialGlobTrace^{(i)} + B_h^{\polyDegMax}\phi_h.
\end{equation}
We use $\trialGlobTrace^{(0)} = 0$ for the initial iteration.

\subsection{Main result}
\label{section: multigrid main results}

\begin{thm}[Convergence bound $W$-cycle]
  \label{theorem: convergence bound non-inherited W-cycle}
  For all $j=2,\ldots, \nLevel$, set $\sigma_j \coloneqq \frac{k_j}{k_{j-1}}$.
  Then, assuming that $\nSmooth$ is large enough, the operator $I - B_h^{\polyDegMax}A_h^{\polyDegMax}$ of \eqref{eq: non-inherited multigrid iterative relation} has spectral radius $\rho$ bounded by
  \begin{equation}\label{eq p-mg convergence rate basic}
    \rho(I - B_h^{\polyDegMax} A_h^{\polyDegMax})\le 1 - \frac{\nSmooth}{C\max_j(\sigma_j^{\frac{5}{4}}k_j^{\frac{7}{4}}) + \nSmooth},
  \end{equation}
  where $C>0$ is some constant independent of $h, \nLevel, \nSmooth$ and the $k_j$. In particular, if $k_j \lesssim k_{j-1}$ for all $j=2,\ldots,\nLevel$, then
  \begin{equation}
  	\label{eq: p-mg convergence rate under bounded polynomial coarsening}
  	  \rho(I - B_h^{\polyDegMax} A_h^{\polyDegMax})\le 1 - \frac{\nSmooth}{C\polyDegMax^{\frac{7}{4}} + \nSmooth},
  \end{equation}
\end{thm}

\begin{rem}[Preconditioning]
  The framework by Bramble, Pasciak and Xu \cite{Bramble.Pasciak.ea:91}, which we use to prove the above theorem, additionally provides bounds for the variable $V$-cycle as a preconditioner.
\end{rem}

\begin{rem}[Relation with introductory statement \eqref{eq: HHO convergence result}]
  The contraction factor $\varrho(\nSmooth)$ from the introduction corresponds to $\rho(I - B_h^{\polyDegMax} A_h^{\polyDegMax})$.
  To obtain equation \ref{eq: HHO convergence result}, which is convenient for comparison with \eqref{eq: DG convergence result}, we use \eqref{eq: p-mg convergence rate under bounded polynomial coarsening} and note that
  \[
  1 - \frac{\nSmooth}{Ck^{\frac{7}{4}} + \nSmooth } = \frac{Ck^{\frac{7}{4}}}{Ck^{\frac{7}{4}} + \nSmooth } \le \frac{Ck^{\frac{7}{4}}}{1 + \nSmooth},
  \]
  where in the last step we assume $Ck^{\frac{7}{4}} \ge 1$.
\end{rem}

  \begin{rem}[Design of $p$-multigrid solvers]
    As we will detail in the following numerical section, finely tracking the dependency on $\sigma_j$ in \eqref{eq p-mg convergence rate basic} suggests the design of $p$-multigrid solvers different from the ones usually found in the literature.
\end{rem}

\subsection{Numerical results}

We take the unit square $\Omega = (0, 1)^2$ and we consider the continuous problem \eqref{eq: continuous Poisson problem} with right-hand side such that the solution is $u(x,y) = \sin(\pi x)\sin(\pi y)$.
We consider its HHO solution on the family of triangular meshes shown in Figure \ref{figure: triangular meshes}.

Let $e_i = \trialGlobTrace - \trialGlobTrace^{(i)}$, $i\ge 0$, denote the error between the discrete solution $A_h^k\trialGlobTrace = \phi_h$ from equation \eqref{eq: multigrid problem on level j} and the $i$-th multigrid iteration $\trialGlobTrace^{(i)}$. It is well known that, in any norm $\|\cdot\|$, the error has an asymptotic reduction rate
$$
\lim_{i\rightarrow\infty}\frac{\|e_{i+1}\|}{\|e_i\|} \le \rho(I - B_h^{\polyDegMax} A_h^{\polyDegMax}),
$$
with equality unless $\trialGlobTrace$ is orthogonal to the eigenspace of the largest eigenvalue of $I - B_h^{\polyDegMax} A_h^{\polyDegMax}$.
Let $\|\testGlobTrace\|_{A_h^k} \coloneqq \|A_h^k\testGlobTrace\|_\facesUnion$ for any $\testGlobTrace\in \hhoSpaceTrace_h^k$.
Then, $\|e_i\|_{A_h^k} = \|\phi_h - A_h^k\trialGlobTrace^{(i)}\|_\facesUnion$ and so the error is measured by the residual error. This is a natural choice since $\trialGlobTrace$ is unknown in general. We estimate the spectral radius by
$$
\rho^{\text{exp}} \coloneqq \frac{\|e_{n+1}\|_{A_h^k}}{\|e_n\|_{A_h^k}},
$$
where $n$ is the last iteration of the numerical experiment.

In the multigrid algorithm \ref{algorithm: general multigrid cycle}, we apply a Richardson smoother $\smoother = \frac{1}{\rho(A_h^{k_j})} I$ at each level $j\ge2$, and we select $\nSmooth$ as small as possible to still yield convergence for all of the tests considered.

Table \ref{table: tri-mesh comparison geometric polynomial coarsening} studies the convergence rate under $hp$-refinement for $\sigma_j = 2$ on all levels $j=2,\ldots, \nLevel$. Indeed, we see that the convergence rate as measured by both the number of iterations and $\rho^{\text{exp}}$ is independent of $h$ and deteriorates with increasing $k$.

Common designs for the polynomial degrees in $p$-multigrid algorithms use arithmetic- or geometric coarsening, for example $k_j = 9, 7, 5, 3, 1$ or $k_j = 8, 4, 2, 1$. However, Theorem \ref{theorem: convergence bound non-inherited W-cycle} suggests that to optimize the convergence rate, one should select the $k_j$ so that $\max_j(\sigma_j^\frac{5}{4} k_j^\frac{7}{4})$ is minimized.
This motivates instead a design of the type $k_j = 8, 7, 4, 1$ or $k_j = 9, 8, 5, 1$.
Table \ref{table: tri-mesh comparison convergence rate of p-design} compares the convergence rates of a few different algorithmic designs.
Indeed, the two coarsening choices $k_j = 8, 7, 6, 5, 4, 3, 2, 1$ and $k_j = 8, 7, 4, 1$ yield the same measured convergence rate $\rho^{\text{exp}}$.
We furthermore see, in accordance with the theory, that the more common choice $k_j = 8, 4, 2, 1$ yields a far worse convergence rate than obtained for the first two choices of the $k_j$. To the best of our knowledge, $p$-multigrid designs to minimize an expression of the form $\max_j(\sigma_j^x k_j^y)$, $x,y>0$, have never been suggested before. Of course, in the design of an efficient solver, the amount of work that goes into each iteration must also be taken into consideration.

\begin{figure}[ht]
  \centering

  \begin{subfigure}{0.245\textwidth}
    \includegraphics[width=\linewidth]{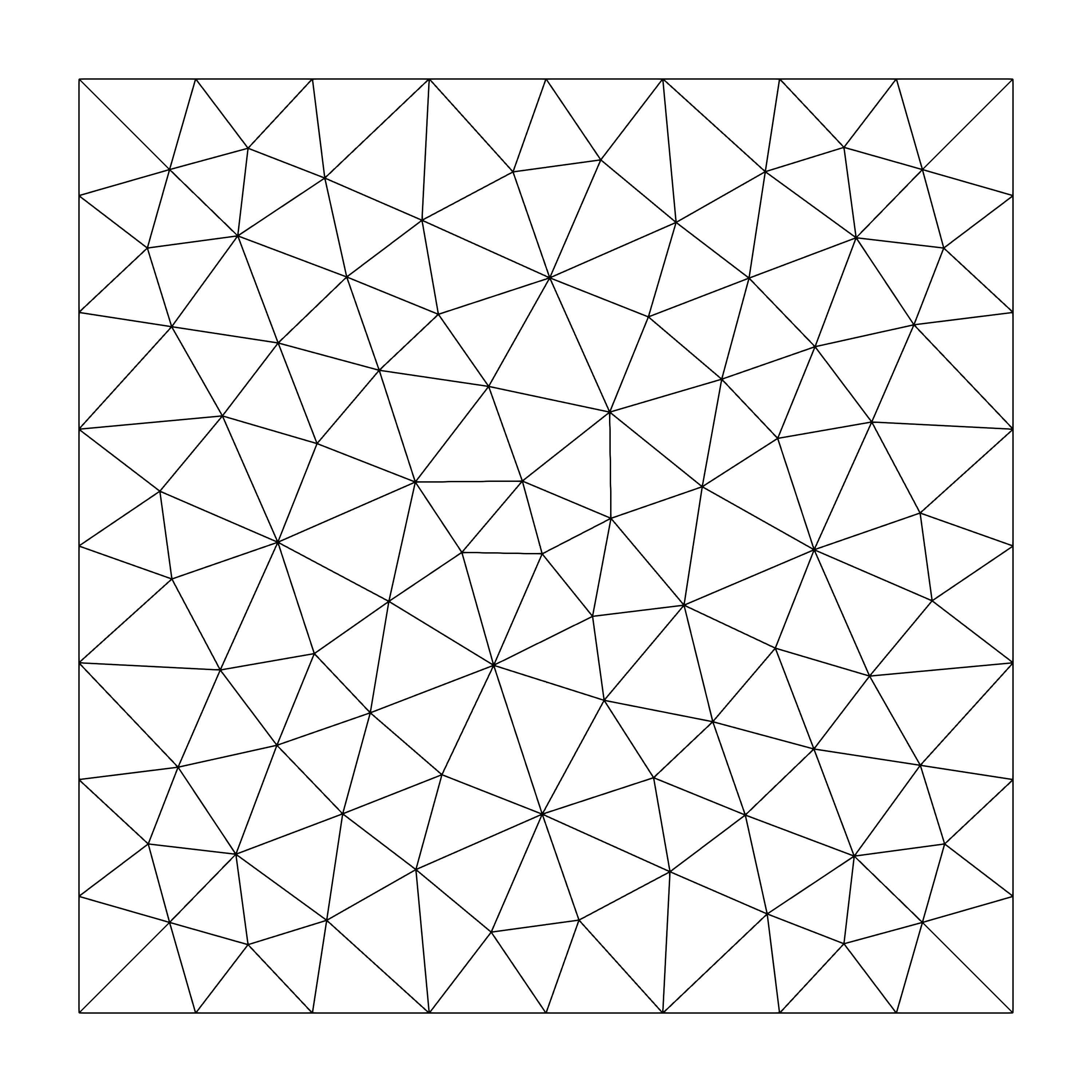}
    \caption{$h=1.68e-1$}
  \end{subfigure}
  \begin{subfigure}{0.245\textwidth}
    \includegraphics[width=\linewidth]{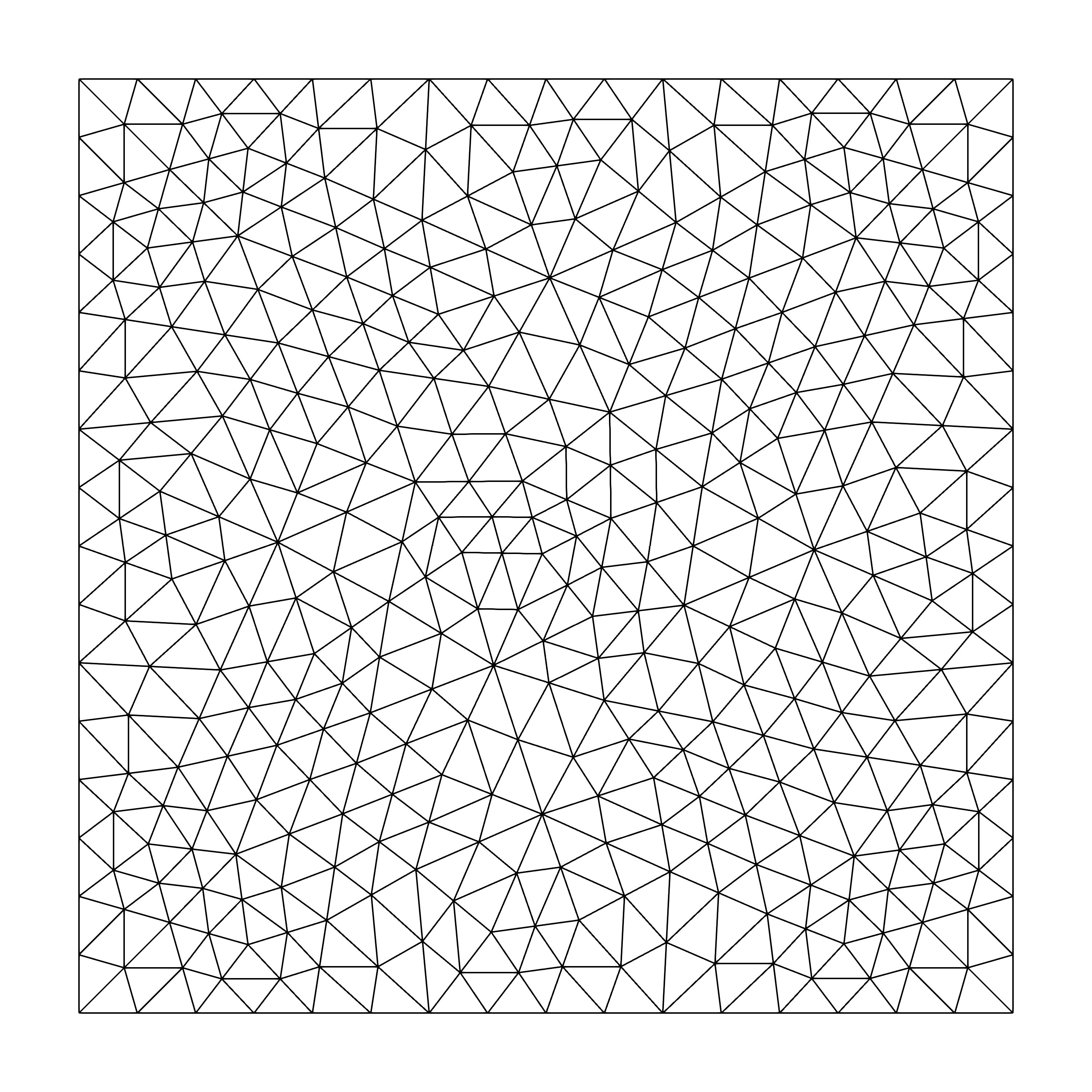}
    \caption{$h=8.38e-2$}
  \end{subfigure}
  \begin{subfigure}{0.245\textwidth}
    \includegraphics[width=\linewidth]{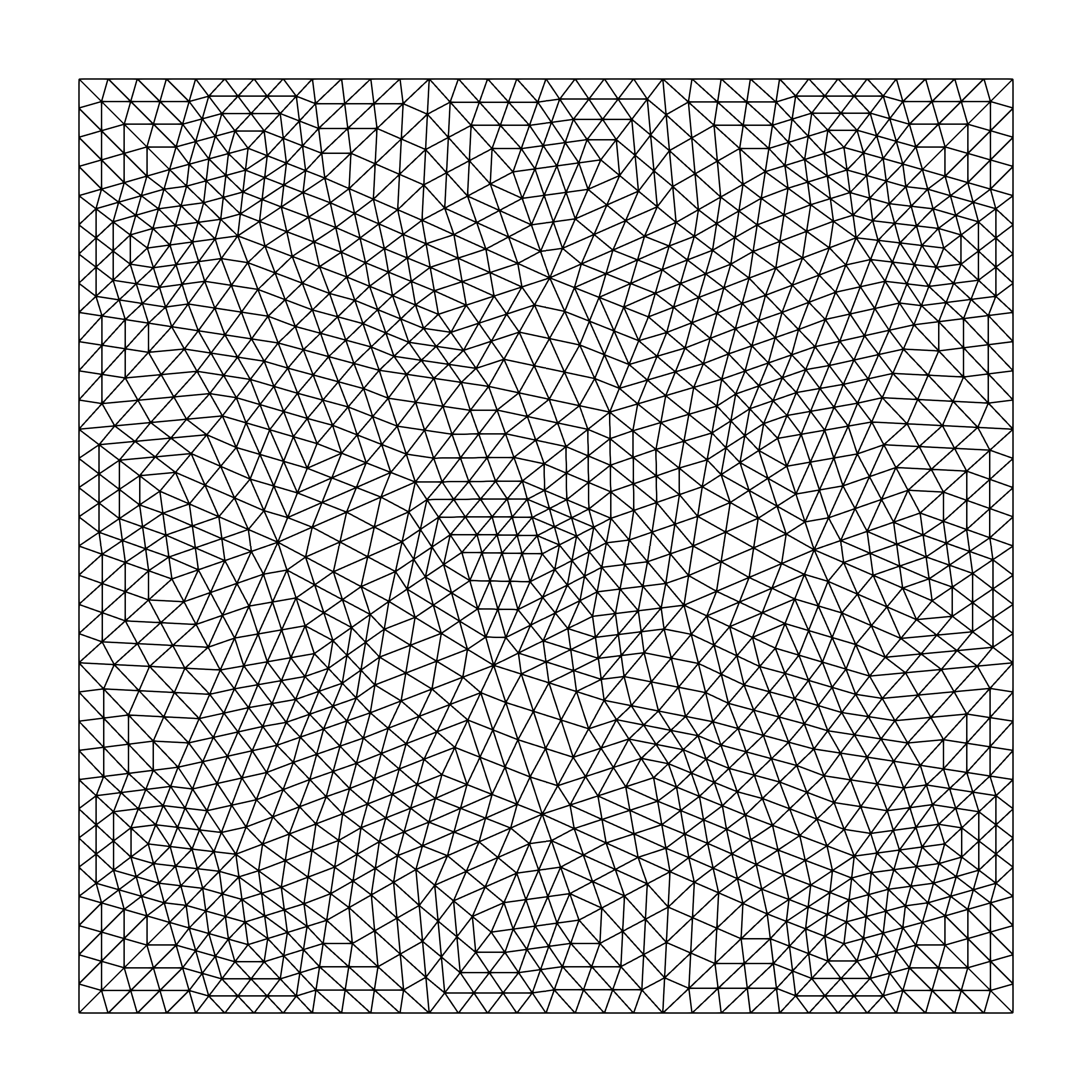}
    \caption{$h=4.19e-2$}
  \end{subfigure}
  \begin{subfigure}{0.245\textwidth}
    \includegraphics[width=\linewidth]{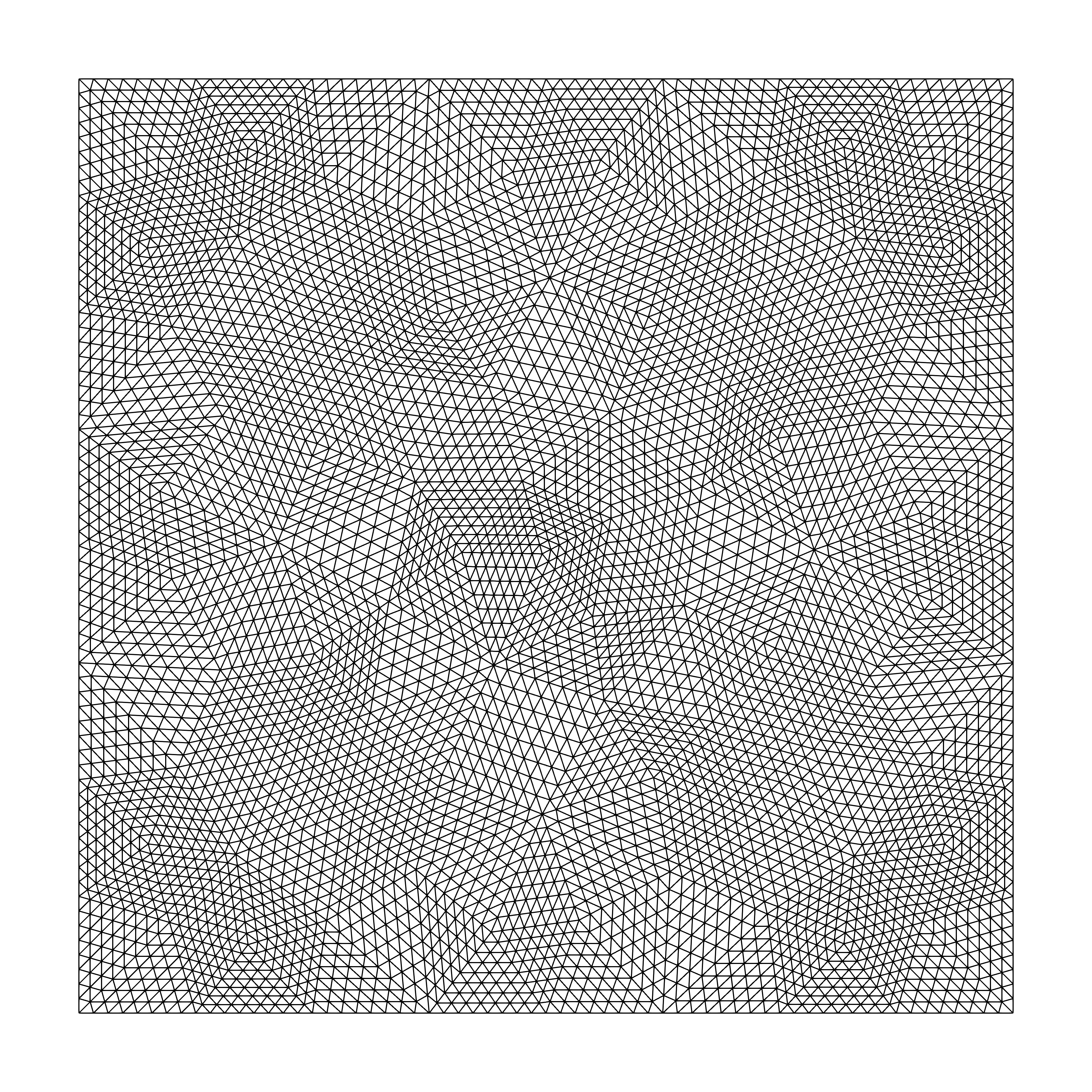}
    \caption{$h=2.10e-2$}
  \end{subfigure}
  \caption{Triangular meshes of similar quality (i) - (iv)}
  \label{figure: triangular meshes}
\end{figure}

\begin{table}[htb!]
  \centering
  \begin{tabular}{| c || c | c | c | c |}
    \hline
    No.\@ it / $\rho^{\text{exp}}$ & tri-mesh (i) &  tri-mesh (ii) &  tri-mesh (iii) &  tri-mesh (iv)\\
    \hline
    \hline
    $k_j = 2, 1$ & 12 / 0.16 & 12 / 0.17 & 12 / 0.18 & 12 / 0.18\\
    \hline
    $k_j = 4, 2, 1$ & 14 / 0.25 & 14 / 0.27 & 14 / 0.27 & 13 / 0.27\\
    \hline
    $k_j = 8, 4, 2, 1$ & 23 / 0.47 & 26 / 0.51 & 26 / 0.51 & 26 / 0.50\\
    \hline
  \end{tabular}
  \caption{We state the number of iterations until the normalized residual error has $\frac{\|e_i\|_{A_h^k}}{\|\trialGlobTrace\|_{A_h^k}} <1e-10$ as well as the observed $\rho^{\text{exp}}$. The multigrid is the symmetric $W$-cycle with $\nSmooth = 8$ from Algorithm \ref{algorithm: general multigrid cycle} which applies the Richardson iterator $\smoother = \frac{1}{\rho(A_h^{k_j})} I$ as smoother.}
  \label{table: tri-mesh comparison geometric polynomial coarsening}
\end{table}

\begin{table}[htb!]
  \centering
  \begin{tabular}{| c || c | c | c | c |}
    \hline
    No.\@ it / $\rho^{\text{exp}}$ & tri-mesh (i) &  tri-mesh (ii) &  tri-mesh (iii) &  tri-mesh (iv)\\
    \hline
    \hline
    $k_j = 8, 7, 6, 5, 4, 3, 2, 1$ & 10 / 0.22 & 10 / 0.25 & 11 / 0.25 & 11 / 0.26\\
    \hline
    $k_j = 8, 7, 4, 1$ & 11 / 0.22 & 12 / 0.24 & 12 / 0.25 & 12 / 0.25\\
    \hline
    $k_j = 8, 4, 2, 1$ & 20 / 0.47 & 22 / 0.51 & 22 / 0.50 & 22 / 0.49\\
    \hline
  \end{tabular}
  \caption{We state the number of iterations until the normalized residual error has $\frac{\|e_i\|_{A_h^k}}{\|\trialGlobTrace\|_{A_h^k}} <1e-10$. The multigrid is the symmetric $W$-cycle with $\nSmooth = 8$ from Algorithm \ref{algorithm: general multigrid cycle} which applies the Richardson iterator $\smoother = \frac{1}{\rho(A_h^{k_j})} I$ as smoother.}
  \label{table: tri-mesh comparison convergence rate of p-design}
\end{table}


\section{Optimal $hp$-error estimates} \label{section: hp error estimates (analysis)}

\subsection{Preliminary lemmas}

\begin{lem}[Properties of the stabilization]\label{lem: stabilization consistency}
  The stabilization in \eqref{eq: standard HHO stabilization} satisfies properties \eqref{eq: stabilization: trace - cell dof bound} and \eqref{eq: stabilization consistency}.
\end{lem}

\begin{proof}
  \emph{Proof of \eqref{eq: stabilization: trace - cell dof bound}.}
  Inserting $\pm \big(\projLLoc\hor\testLoc - \projLTraceLoc\hor\testLoc\big)$ into the norm and using a triangle inequality, we get
  \begin{align*}
    \|\testLocTrace - \testLocCell\|_{\partial T}
    &\le \big\|\testLocTrace - \projLTraceLoc\hor\testLoc - \testLocCell + \projLLoc\hor\testLoc\big\|_{\partial T}
    \\
    &\quad
    + \big\|\projLTraceLoc\hor\testLoc - \projLLoc\hor\testLoc\big\|_{\partial T}\\
    \overset{\eqref{eq:difference operators},\eqref{eq: standard HHO stabilization},\eqref{eq:non condensed norms}}&=
    \errFacLoc^{\frac{1}{2}}\sNormNonCondLoc{\testLoc}
    + \big\|\projLTraceLoc\big(\hor\testLoc - \projLLoc\hor\testLoc\big)\big\|_{\partial T}\\
    \overset{\eqref{eq: stability of L2-projection on local traces}}
    &\le \errFacLoc^{\frac{1}{2}}\sNormNonCondLoc{\testLoc} + \big\|\hor\testLoc - \projLLoc\hor\testLoc\big\|_{\partial T}\\
    \overset{\eqref{eq: approximation L2-projection on trace FIRST ORDER}}
    &\lesssim \errFacLoc^{\frac{1}{2}}\Big(\sNormNonCondLoc{\testLoc} + \|\nabla\hor\testLoc\|_T\Big)\\
    \overset{\eqref{eq: L2-norm hor <= L2-norm gradHor}}
    &\le \errFacLoc^{\frac{1}{2}}\Big(\sNormNonCondLoc{\testLoc} + \big\|\gradHor\testLoc\big\|_T\Big),
  \end{align*}
  where we have additionally used the fact that the trace of $\projLLoc\hor\testLoc$ on $\partial T$ is in $\IP^k(\partial T) \subset \IP^k(\facesLoc)$ to write $\projLLoc\hor\testLoc = \projLTraceLoc \projLLoc\hor\testLoc$ in the second step.
  \smallskip\\
  \noindent\emph{Proof of \eqref{eq: stabilization consistency}.}
  Recalling the definition \eqref{eq:non condensed norms} of $\sNormNonCondLoc{\cdot}$,
  expanding first $\sNonCondLoc$ and then the difference operators according to their respective definitions \eqref{eq: standard HHO stabilization} and \eqref{eq:difference operators}, and additionally using $\hor\interpolLoc w \overset{\eqref{eq: hor commutation with interpolation}} = \projEllipticLocPlus w$, we get
  \begin{align*}
    \sNormNonCondLoc{\interpolLoc w}
    &= \errFacInvLoc^{\frac{1}{2}} \big\|\projLTraceLoc\big(w - \projEllipticLocPlus w\big) - \projLLoc\big(w - \projEllipticLocPlus w\big)\big\|_{\partial T}\\
    &= \errFacInvLoc^{\frac{1}{2}}\big\|\projLTraceLoc\big( w - \projEllipticLocPlus - \projLLoc\big(w - \projEllipticLocPlus w\big)\big)\big\|_{\partial T}\\
    \overset{\eqref{eq: stability of L2-projection on local traces}}&\le \errFacInvLoc^{\frac{1}{2}}\big\|w - \projEllipticLocPlus w - \projLLoc(w - \projEllipticLocPlus w)\big\|_{\partial T}\\
    \overset{\eqref{eq: approximation L2-projection on trace FIRST ORDER}}&\lesssim \left\|\nabla w - \nabla\projEllipticLocPlus w\right\|_{T}
    \overset{\eqref{eq: approximation elliptic-projection in first Sobolev norm on cell}}\lesssim \errFacLoc^{m-1}\left\|w\right\|_{m, T}.
    \qedhere
  \end{align*}
\end{proof}

\begin{lem}[Continuity of the normal trace component]
  For any $v\in H^2(\Omega)$ and any $\testGlobTrace\in \hhoSpaceTrace_h^k$, it holds that
  \begin{equation}
    \label{eq: cancellation via normal trace component}
    \sumT (\testLocTrace, \nabla v\cdot\normalT)_{\partial T} = 0.
  \end{equation}
\end{lem}
\begin{proof}
  Consequence of \cite[Corollary 1.19]{Di-Pietro.Droniou:20} applied to $\genericVec = \nabla v \in H^1(\Omega)^d$ and $(\varphi_F)_{F \in \faces}$ such that $\varphi_F = \testGlobTrace|_F$ for all $F \in \faces$ (notice that, by definition \eqref{eq:global hho spaces} of $\hhoSpaceTrace_h^k$, $\varphi_F = 0$ for all $F \subset \partial \Omega$).
\end{proof}

For skeletal $L^2$-estimates we shall need the following lemma.
\begin{lem}[Energy bound element component] \label{lemma: bound grad vT in terms of energy}
  For all $T \in \mesh$ and all $\testLoc\in\hhoSpaceLoc$, it holds
  \begin{equation}
    \label{eq: bound nabla cell dof in terms of energy norm}
    \|\nabla \testLocCell\|_T \lesssim \sqrt[]{k}\aNormNonCondLoc{\testLoc}.
  \end{equation}
\end{lem}
\begin{proof}
  Using the definition \eqref{eq: gradHor definition} of the discrete gradient $\gradHor$ with $\trialLoc = \testLoc$ and $\genericVec = \nabla \testLocCell$, we have
  \begin{align*}
    \|\nabla \testLocCell\|^2_T &= \big(\gradHor\testLoc, \nabla \testLocCell\big)_T - \big(\testLocTrace - \testLocCell, \nabla\testLocCell\cdot\normalT\big)_{\partial T}\\
    &\le  \|\gradHor\testLoc\|_T \|\nabla v_T\|_T + \|\testLocTrace - \testLocCell\|_{\partial T} \|\nabla \testLocCell\|_{\partial T}\\
    \overset{\eqref{eq: stabilization: trace - cell dof bound}, \eqref{eq: discrete inverse trace inequality}}&\lesssim \|\gradHor\testLoc\|_T \|\nabla v_T\|_T + \errFacLoc^{\frac{1}{2}}\aNormNonCondLoc{\testLoc} \frac{k-1}{h_T^{\frac{1}{2}}} \|\nabla \testLocCell\|_{T}\\
    &\le \|\nabla \testLocCell\|_T\left(\left\|\gradHor\testLoc\right\|_T  + \sqrt{k}\aNormNonCondLoc{\testLoc}\right).
  \end{align*}
  Simplifying, the conclusion follows after noticing that $1 \le \sqrt{k}$.
\end{proof}

\subsection{Non-condensed error analysis}

In the notation of Theorems \ref{theorem: non-condensed energy error estimate} and \ref{theorem: non-condensed L2-error estimate}, let $\errGlob\in\hhoSpaceGlob$ be the approximation error given by
\begin{equation}\label{eq:errGlob}
  \errGlob \coloneqq \trialGlob - \interpolGlob{u}.
\end{equation}

\begin{proof}[Proof of Theorem \ref{theorem: non-condensed energy error estimate}]
    Using the linearity of $\aNonCondGlob$ in its first argument followed by the fact that
  $\trialGlob$ solves problem \eqref{eq: global problem discretized non-condensed}, we have, for any $\testGlob\in\hhoSpaceGlob$,
  \begin{align}
    \aNonCondGlobArg{\errGlob}{\testGlob} &= (f, \testGlobCell)_\Omega -  \aNonCondGlobArg{\interpolGlob{u}}{\testGlob} \notag\\
    \label{eq: consistency error}
    &= (-\Delta u, \testGlobCell)_\Omega -  \aNonCondGlobArg{\interpolGlob{u}}{\testGlob},
  \end{align}
    where we have used the fact that  $f = -\Delta u$ almost everywhere in $\Omega$ in the second step.
  Integrating the first term by parts inside each element, we get
  \begin{equation}\label{eq:IBP}
    \begin{aligned}
      (-\Delta u, \testGlobCell)_\Omega &= \sumT \big[(\nabla u, \nabla \testLocCell)_T + (\nabla u \cdot \normalT, -\testLocCell)_{\partial T}\big]\\
      &= \sumT \big[(\projLLocVec\nabla u, \nabla \testLocCell)_T + (\nabla u \cdot \normalT, \testLocTrace - \testLocCell)_{\partial T}\big],
    \end{aligned}
  \end{equation}
  where, in the second step, we have inserted $\testLocTrace$ into the boundary term using \eqref{eq: cancellation via normal trace component}.
  Expanding the second term in \eqref{eq: consistency error} according to \eqref{eq:aNonCondGlob}, on the other hand, we have
  \begin{equation}
    \label{eq: standard derivation p.11}
    \begin{aligned}
      &\aNonCondGlobArg{\interpolGlob{u}}{\testGlob}
      \\
      &\quad
      \begin{aligned}[t]
        &= \sumT\Big[\big(\gradHor\interpolLoc{u}, \gradHor\testLoc\big)_T + \sNonCondLocArg{\interpolLoc{u}}{\testLoc}\Big]\\
        \overset{\eqref{eq: gradHor commutation with interpolation}} &= \sumT\Big[\big(\projLLocVec\nabla u, \gradHor\testLoc\big)_T + \sNonCondLocArg{\interpolLoc{u}}{\testLoc}\Big]\\
        \overset{\eqref{eq: gradHor definition}}&= \sumT\Big[\big(\projLLocVec\nabla u, \nabla\testLocCell\big)_T + (\projLLocVec\nabla u \cdot \normalT, \testLocTrace - \testLocCell)_{\partial T} + \sNonCondLocArg{\interpolLoc{u}}{\testLoc}\Big].
      \end{aligned}
    \end{aligned}
  \end{equation}
    Inserting \eqref{eq:IBP} and \eqref{eq: standard derivation p.11}
  back into \eqref{eq: consistency error} gives
  \begin{equation}\label{eq: energy error estimate intermediate}
    \begin{aligned}
      \aNonCondGlobArg{\errGlob}{\testGlob} &= \sumT \left[
        (\nabla u - \projLLocVec\nabla u)\cdot\normalT, \testLocTrace - \testLocCell\big)_{\partial T} - \sNonCondLocArg{\interpolLoc{u}}{\testLoc}
        \right]
      \\
      & \le  \left(
      \sumT \frac{h_T}{k} \|(\nabla u - \projLLocVec\nabla u)\|_{\partial T}^2
      \right)^{\frac{1}{2}} \left(
      \sumT\frac{k}{h_T} \|\testLocTrace - \testLocCell\|^2_{\partial T}
      \right)^{\frac{1}{2}}
      \\
      &\tab + \left(
      \sumT |\interpolLoc{u}|_{\sNonCondLoc}^2
      \right)^{\frac{1}{2}} \left(
      \sumT|\testLoc|_{\sNonCondLoc}^2
      \right)^{\frac{1}{2}}
      \\
      \overset{\eqref{eq: approximation L2-projection on trace}, \eqref{eq: stabilization: trace - cell dof bound}, \eqref{eq: stabilization consistency}}
      & \lesssim \left(
      \sumT \errFacLoc^{2(m-1)}\|u\|_{m, T}^2
      \right)^{\frac{1}{2}} \left(
      \sumT \aNormNonCondLoc{\testLoc}
      \right)^{\frac{1}{2}}
      \\
      & \lesssim  \errFac^{m-1}\|u\|_{m,h} \aNormNonCondGlob{\testGlob},
    \end{aligned}
  \end{equation}
  where, in the second step, we used discrete Hölder inequalities along with $\|\normalT\|_{L^\infty(\partial T)^d} = 1$.
  Choosing $\testGlob = \errGlob$, recalling \eqref{eq: aNormNonCondGlob}, and simplifying finishes the proof.
\end{proof}

\begin{proof}[Proof of Theorem \ref{theorem: non-condensed L2-error estimate}]
  Our approach follows the one in \cite{Aghili.Di-Pietro.ea:17}. Let $z\in H_0^1(\Omega)$ solve
  $$
  (\nabla z, \nabla v)_\Omega = (\errGlobCell, v)_\Omega \qquad \forall v\in H^{1}_0(\Omega).
  $$
  Since $\errGlobCell\in L^2(\Omega)$, elliptic regularity yields $z\in H^2(\Omega)$ and we have $-\Delta z  = \errGlobCell$ almost everywhere in $\Omega$.
  We have, integrating by parts on each element and using  \eqref{eq: cancellation via normal trace component} as in \eqref{eq:IBP},
    \begin{equation}
    \label{eq: non-condensed L2-error error expression}
    \|\errGlobCell\|_\Omega^2
    = (\errGlobCell, -\Delta z)_\Omega
    = \sumT \Big[(\nabla \errLocCell, \projLLocVec\nabla z)_T + (\errLocTrace-\errLocCell, \nabla z \cdot \normalT)_{\partial T}\Big].
    \end{equation}
  We will add and subtract two different expressions for $\aNonCondGlobArg{\errGlob}{\interpolGlob{z}}$ to the above error. The minus expression is obtained as for \eqref{eq: standard derivation p.11}:
  \begin{multline}
    \label{eq: non-condensed L2-error minus expression}
    \aNonCondGlobArg{\errGlob}{\interpolGlob{z}}
    \\
    \overset{\eqref{eq: gradHor commutation with interpolation}, \eqref{eq: gradHor definition}}= \sumT\big[\big(\nabla\errLocCell, \projLLocVec\nabla z\big)_T + (\errLocTrace-\errLocCell, \projLLocVec\nabla z \cdot \normalT)_{\partial T} + \sNonCondLocArg{\errLoc}{\interpolLoc{z}}\big].
  \end{multline}
  The plus expression is obtained by using the definition \eqref{eq:errGlob} of the error together with the linearity of $\aNonCondGlob$ in its first argument to write
  \begin{equation}
    \label{eq: non-condensed L2-error plus expression}
    \begin{aligned}
      &\aNonCondGlobArg{\errGlob}{\interpolGlob{z}}
      \\
      &
      \begin{aligned}[t]
        &=\aNonCondGlobArg{\trialGlob}{\interpolGlob{z}} - \aNonCondGlobArg{\interpolGlob{u}}{\interpolGlob{z}}
        \\
        \overset{\eqref{eq: global problem discretized non-condensed}, \eqref{eq:aNonCondGlob}, \eqref{eq: gradHor commutation with interpolation}}&=
        \sumT\left[
          (f, \projLLoc z)_T - (\projLLocVec\nabla u, \projLLocVec\nabla z)_T - \sNonCondLocArg{\interpolLoc{u}}{\interpolLoc{z}}
          \right]
        \\
        &= \sumT\left[
          (f,  \projLLoc z)_T - (f, z)_T + (\nabla u, \nabla z)_T - (\projLLocVec\nabla u, \projLLocVec\nabla z)_T
          - \sNonCondLocArg{\interpolLoc{u}}{\interpolLoc{z}}
          \right]
        \\
        &= \sumT\left[
          (f - \projLLoc f, \projLLoc z - z)_T + (\nabla u-\projLLocVec\nabla u, \nabla z-\projLLocVec\nabla z)_T
          - \sNonCondLocArg{\interpolLoc{u}}{\interpolLoc{z}}
          \right],
      \end{aligned}
    \end{aligned}
  \end{equation}
  where in the third step we added $0 \overset{\eqref{eq: continuous Poisson problem}} = -(f, z)_\Omega + (\nabla u, \nabla z)_\Omega$,
    while the conclusion follows from the fact that $\projLLoc z - z$ and $\nabla z-\projLLocVec\nabla z$ are by definition $L^2$-orthogonal to $\projLLoc f \in \IP^k(T)$ and $\projLLocVec\nabla u \in \IP^k(T)^d$, respectively.
  Subtracting \eqref{eq: non-condensed L2-error minus expression} and adding \eqref{eq: non-condensed L2-error plus expression} to \eqref{eq: non-condensed L2-error error expression}, we get
  \begin{equation}
    \label{eq: non-condensed L2-error combined expression}
    \begin{aligned}
      &\|\errGlobCell\|_\Omega^2
      = \underbrace{%
        \sumT \Big[\big(\errLocTrace-\errLocCell, (\nabla z - \projLLocVec\nabla z) \cdot \normalT\big)_{\partial T} - \sNonCondLocArg{\errLoc}{\interpolLoc{z}}\Big]
      }_{\mathfrak{T}_1}
      \\
      &\quad
      \underbrace{%
        + \sumT\big[(f - \projLLoc f, \projLLoc z - z)_T + (\nabla u-\projLLocVec\nabla u, \nabla z-\projLLocVec\nabla z)_T - \sNonCondLocArg{\interpolLoc{u}}{\interpolLoc{z}}\big].
      }_{\mathfrak{T}_2}
    \end{aligned}
  \end{equation}
    Proceeding as for the first inequality in \eqref{eq: energy error estimate intermediate} with substitutions $u \gets z$, $\testGlob \gets \errGlob$, and $m \gets 2$, we get
  \[
    \mathfrak{T}_1
    \lesssim \frac{h}{k} \aNormNonCondGlob{\errGlob}\|z\|_{2,\Omega}
    \overset{\eqref{eq: theorem non-condensed energy error estimate}}\lesssim \errFac^{m}\|u\|_{m,h}\|z\|_{2,\Omega}.
  \]
  Moving on to the second sum and applying again Cauchy--Schwarz inequalities, we get
  \begin{align*}
    \mathfrak{T}_2
    &\le \left(
    \sumT \|f - \projLLoc f\|^2_T
    \right)^{\frac{1}{2}} \left(
    \sumT \|z - \projLLoc z\|^2_T
    \right)^{\frac{1}{2}}
    + \left(
    \sumT \|\nabla u-\projLLocVec\nabla u\|^2_T
    \right)^{\frac{1}{2}}\left(
    \sumT \|\nabla z-\projLLocVec\nabla z\|^2_T
    \right)^{\frac{1}{2}}\\
    &\quad + \left(
    \sumT \sNormNonCondLoc{\interpolLoc{u}}^2
    \right)^{\frac{1}{2}} \left(
    \sumT \sNormNonCondLoc{\interpolLoc{z}}^2
    \right)^{\frac{1}{2}}
    \\
    \overset{\eqref{eq: approximation L2-projection on cell}, \eqref{eq: stabilization consistency}}&\lesssim
    \left(
    \sumT \errFacLoc^{2(m-2)}\|f\|^2_{m-2,T}
    \right)^{\frac{1}{2}} \left(
    \sumT \errFacLoc^{4}\|z\|^2_{2,T}
    \right)^{\frac{1}{2}}
    \\
    &\quad + \left(
    \sumT \errFacLoc^{2(m-1)}\|u\|^2_{m,T}
    \right)^{\frac{1}{2}}\left(
    \sumT \errFacLoc^{2}\|z\|^2_{2,T}
    \right)^{\frac{1}{2}}
    \\
    &\lesssim \errFac^{m}\left( \|f\|_{m-2,h} + \|u\|_{m,h} \right)\|z\|_{2,\Omega},
  \end{align*}
  Substituting the bounds for $\mathfrak{T}_1$ and $\mathfrak{T}_2$ into \eqref{eq: non-condensed L2-error combined expression}, we obtain
  \begin{align*}
    \|\errGlobCell\|_\Omega^2\lesssim \errFac^{m}\left( \|f\|_{m-2,h} + \|u\|_{m,h} \right) \|z\|_{2,\Omega}.
  \end{align*}
    Using elliptic regularity to write $\|z\|_{2,\Omega} \overset{\eqref{eq: elliptic regularity}} \lesssim \|\errGlobCell\|_\Omega$ and simplifying, the conclusion follows.
\end{proof}

\subsection{Condensed error analysis}\label{section: condensed error analysis}
To prove the condensed error estimates of Section \ref{section: hho main results}, we will use the non-condensed error estimates from the last section and the following new error estimates for the local condensation operators $\recLoc^k$ and $\vLoc$.

\begin{lem}[Error estimate on $\vLoc$ to first order]
  For any $T\in\mesh$ and any $g\in L^2(T)$, let $u_T^P\in H_0^1(T)$ be the solution to
  \begin{equation}
    \label{eq: u_T^P definition}
    (\nabla u_T^P, \nabla v)_T = (g, v)_T \qquad \forall v\in H_0^{1}(T).
  \end{equation}
  Then, it holds,
  \begin{align}
    \label{eq: V_T^k energy error estimate}
    \aNormNonCondLoc{(\vLoc g - \projLLoc u_T^P, 0)} &\lesssim \frac{h_T}{k}\|g\|_{T},
    \\
    \label{eq: V_T^k L2 error estimate}
    \|\vLoc g - \projLLoc u_T^P\|_T &\lesssim\errFacLoc^{2}\|g\|_{T}.
  \end{align}
\end{lem}

\begin{proof}
  By elliptic regularity, we have $u_T^P\in H^2(T)$.
  It suffices to notice that, by definition \eqref{eq:vLoc}, $(\vLoc g , 0)\in\hhoSpaceCell_T^k\times\{0\}$ is the discretized solution to problem \eqref{eq: u_T^P definition} with the HHO scheme on the mesh $(T, \facesLoc)$. Thus, the non-condensed error estimates of \eqref{eq: theorem non-condensed energy error estimate} and \eqref{eq: theorem non-condensed L2 error estimate} yield, respectively,
  \begin{align*}
    \aNormNonCondLoc{(\vLoc g - \projLLoc u_T^P, 0)} &= \aNormNonCondLoc{(\vLoc g , 0) - \interpolLoc{u_T^P}} \lesssim\frac{h_T}{k}\|u_T^P\|_{2, T},
    \\
    \|\vLoc g - \projLLoc u_T^P\|_T &\lesssim\errFacLoc^{2}\big(\|u_T^P\|_{2, T} + \|g\|_{T}\big).
  \end{align*}
  We then use elliptic regularity to write $\|u_T^P\|_{2, T}\overset{\eqref{eq: elliptic regularity}}  \lesssim \|g\|_{T}$ and conclude.
\end{proof}

\begin{lem}[Error estimate on $\recLoc^k$ to first order]
  For any $T\in\mesh$ and any $w\in H^2(T)$, let $\gamma \coloneqq w|_{\partial T}$ and let $u_T^H\in H^1(T)$ be the solution to
  \begin{subequations}\label{eq: u_T^H}
    \begin{align}
      \label{eq: u_T^H definition}
      (\nabla u_T^H, \nabla v)_T &= 0 \qquad \forall v\in H_0^{1}(T)\\
      \label{eq: u_T^H boundary condition}
      u_T^H|_{\partial T} &= \gamma.
    \end{align}
  \end{subequations}
  Then,
  \begin{align}
    \label{eq: U_T^k energy error estimate}
    \aNormNonCondLoc{(\recLoc^k \projLTraceLoc \gamma - \projLLoc u_T^H, 0)} &\lesssim\frac{h_T}{k}\|u_T^H\|_{2, T},
    \\
    \label{eq: U_T^k L2 error estimate}
    \|\recLoc^k \projLTraceLoc \gamma - \projLLoc u_T^H\|_T &\lesssim\errFacLoc^{2}\|u_T^H\|_{2, T}.
  \end{align}
\end{lem}

\begin{proof}
We decompose $u_T^H$ as $u_T^H = w + u^*$, where $u^*\in H_0^1(T)$ solves
  \[
  (\nabla u^*, \nabla v)_T = -(\Delta w, v)_T \qquad \forall v\in H_0^1(T),
  \]
  and, by elliptic regularity, $u^*\in H^2(T)$. Now, let
  \begin{equation}\label{eq: errLocCell}
    \errLocCell \coloneqq \recLoc^k \projLTraceLoc \gamma - \projLLoc u_T^H.
  \end{equation}
  \noindent\emph{Proof of \eqref{eq: U_T^k energy error estimate}.}
  Notice that, as $\projLTraceLoc u_T^H = \projLTraceLoc\gamma$, we have, for any $\testLocCell\in\hhoSpaceCell_T^k$,
  \[
  \begin{aligned}
    \aNonCondLocArg{(\errLocCell, 0)}{(\testLocCell, 0)}
    \overset{\eqref{eq: U_T^k definition}}
    &= -\aNonCondLocArg{\interpolLoc{u_T^H}}{(\testLocCell, 0)}
    \\
    \overset{\eqref{eq:aNonCondGlob},\eqref{eq: gradHor commutation with interpolation}}
    &= -\big(\projLLocVec\nabla u_T^H, \gradHor(\testLocCell, 0)\big)_T - \sNonCondLocArg{\interpolLoc{u_T^H}}{(\testLocCell, 0)}
    \\
    \overset{\eqref{eq: gradHor definition}}
    &= -(\nabla u_T^H, \nabla\testLocCell)_T
    + (\projLLocVec\nabla u_T^H\cdot\normalT, \testLocCell)_{\partial T}
    - \sNonCondLocArg{\interpolLoc{u_T^H}}{(\testLocCell, 0)}
    \\
    &= \cancel{(\Delta u_T^H, \testLocCell)_T}
    - \big((\nabla u_T^H - \projLLocVec\nabla u_T^H)\cdot\normalT, \testLocCell\big)_{\partial T}
    - \sNonCondLocArg{\interpolLoc{u_T^H}}{(\testLocCell, 0)}
    \\
    &\le \|\nabla u_T^H - \projLLocVec\nabla u_T^H\|_{\partial T}\|\testLocCell\|_{\partial T} + \sNormNonCondLoc{\interpolLoc{u_T^H}}\sNormNonCondLoc{(\testLocCell, 0)}
    \\
    \overset{\eqref{eq: approximation L2-projection on trace}, \eqref{eq: stabilization: trace - cell dof bound}, \eqref{eq: stabilization consistency}}
    &\lesssim \frac{h_T}{k}\|u_T^H\|_{2, T}\aNormNonCondLoc{(\testLocCell, 0)},
  \end{aligned}
  \]
  where
  in the fourth step we integrated by parts and used the fact that $\Delta u_T^H = 0$ almost everywhere in $T$,
  while in the fifth step we used H\"{o}lder and Cauchy--Schwarz inequalities along with $\| \normalT \|_{L^\infty(\partial T)^d} \le 1$.
  We choose $\testLocCell = \errLocCell$ and simplify to conclude.
  \smallskip\\
  \noindent\emph{Proof of \eqref{eq: U_T^k L2 error estimate}.}
  We proceed similarly to the proof of the non-condensed $L^2$-estimate in Theorem \ref{theorem: non-condensed L2-error estimate}.
  Let $z\in H_0^1(T)$ solve
  $$
  (\nabla z, \nabla v)_T = (\errLocCell, v)_T \qquad \forall v\in H_0^1(T).
  $$
  By elliptic regularity, $z\in H^2(T)$ and $-\Delta z  = \errLocCell$ almost everywhere.
  Integrating by parts, we then have
  \begin{equation}
    \label{eq: U_T^k L2-error error expression}
    \|\errLocCell\|_T^2 = (\errLocCell, -\Delta z)_T
    = (\nabla\errLocCell, \nabla z)_T - (\errLocCell, \nabla z\cdot\normalT)_{\partial T}.
  \end{equation}
  We will add and subtract two different expressions for $\aNonCondLocArg{(\errLocCell, 0)}{\interpolLoc{z}}$ to the above error. The minus expression is obtained proceeding as for \eqref{eq: standard derivation p.11}:
  \begin{multline}
    \label{eq: U_T^k L2-error minus expression}
    \aNonCondLocArg{(\errLocCell, 0)}{\interpolLoc{z}}
    \\
    \overset{\eqref{eq:aNonCondGlob},\eqref{eq: gradHor commutation with interpolation}, \eqref{eq: gradHor definition}}
    = (\nabla\errLocCell, \projLLocVec\nabla z)_T - (\errLocCell, \projLLocVec\nabla z \cdot \normalT)_{\partial T} + \sNonCondLocArg{(\errLocCell, 0)}{\interpolLoc{z}}.
  \end{multline}
  The plus expression is obtained by expanding the error component and noting that, as $\projLTraceLoc z = 0$ since $z \in H_0^1(T)$ and $\projLTraceLoc u_T^H = \projLTraceLoc \gamma$,
  \begin{equation}
    \label{eq: U_T^k L2-error plus expression}
    \begin{aligned}
      &\aNonCondLocArg{(\errLocCell, 0)}{\interpolLoc{z}}
      \\
      &\quad
      \begin{aligned}[t]
        \overset{\eqref{eq: errLocCell}}&=
        \aNonCondLocArg{(\recLoc^k \projLTraceLoc \gamma, 0)}{\interpolLoc{z}} - \aNonCondLocArg{(\projLLoc u_T^H, 0)}{\interpolLoc{z}}\\
        \overset{\eqref{eq: U_T^k definition}}
        &= -\aNonCondLocArg{\interpolLoc{u_T^H}}{\interpolLoc{z}}\\
        &= -(\projLLocVec\nabla u_T^H, \projLLocVec\nabla z)_T
        -\sNonCondLocArg{\interpolLoc{u_T^H}}{\interpolLoc{z}}\\
        &= (\nabla u_T^H -\projLLocVec\nabla u_T^H,\nabla z -\projLLocVec\nabla z)_T
        -\sNonCondLocArg{\interpolLoc{u_T^H}}{\interpolLoc{z}},
      \end{aligned}
    \end{aligned}
  \end{equation}
  where, in the last step, we added $0\overset{\eqref{eq: u_T^H definition}} = (\nabla u_T^H, \nabla z)_T$ and used $L^2$-orthogonality.
  Subtracting \eqref{eq: U_T^k L2-error minus expression} and adding \eqref{eq: U_T^k L2-error plus expression} to \eqref{eq: U_T^k L2-error error expression}, we get
  \begin{align*}
    \|\errLocCell\|_T^2 &= -\big(\errLocCell, (\nabla z - \projLLocVec\nabla z \cdot \normalT)\big)_{\partial T} + \sNonCondLocArg{(\errLocCell, 0)}{\interpolLoc{z}}\\
    &\quad + (\nabla u_T^H -\projLLocVec\nabla u_T^H,\nabla z -\projLLocVec\nabla z)_T -\sNonCondLocArg{\interpolLoc{u_T^H}}{\interpolLoc{z}}\\
    \overset{\eqref{eq: stabilization: trace - cell dof bound}, \eqref{eq: approximation L2-projection on trace}, \eqref{eq: stabilization consistency}, \eqref{eq: approximation L2-projection on cell}}
    &\lesssim \frac{h_T}{k}\aNormNonCondLoc{(\errLocCell, 0)}\|z\|_{2,T} + \errFacLoc^{2}\|u_T^H\|_{2,T}\|z\|_{2,T}\\
    \overset{\eqref{eq: U_T^k energy error estimate}}
    &\lesssim \errFacLoc^2\|u_T^H\|_{2, T}\|z\|_{2,T}.
  \end{align*}
  We use elliptic regularity to infer $\|z\|_{2,T}\overset{\eqref{eq: elliptic regularity}} \lesssim \|\errLocCell\|_T$ and simplify to conclude.
\end{proof}

To prove Theorems \ref{theorem: condensed energy error estimate} and \ref{theorem: condensed L2-error estimate}, we first state some common tools. Let $f$ and $u$ be as in these theorems.
For all $T\in\mesh$, let $u_T^P$ be defined by \eqref{eq: u_T^P definition} with $f$ replacing $g$, and let $u_T^H$ be defined by \eqref{eq: u_T^H} with $u$ replacing $w$. Then, both $u|_T$ and $u_T^P + u_T^H$ are weak solutions to the PDE
\[
\begin{alignedat}{2}
  -\Delta \widetilde{u} &= f &\quad& \text{on $T$},
  \\
  \widetilde{u} &= u|_{\partial T} &\quad& \text{on $\partial T$},
\end{alignedat}
\]
and so, almost everywhere,
\begin{equation}
  \label{eq: u|_T = u_T^P + u_T^H}
  u|_T = u_T^P + u_T^H.
\end{equation}
Furthermore, we have
\begin{equation}
  \label{eq: u_T^H help bound}
  \sumT \|u_T^H\|_{2,T}^2 \lesssim \|f\|_\Omega^2,
\end{equation}
as adding and subtracting $u_T^P$ gives
\[
\sumT\|u_T^H\|^2_{2, T}
\lesssim \sumT \left( \|u_T^H + u_T^P\|^2_{2, T} + \|u_T^P\|^2_{2, T} \right)
\overset{\eqref{eq: u|_T = u_T^P + u_T^H}}
= \|u\|^2_{2, \Omega} + \sumT\|u_T^P\|^2_{2, T}
\lesssim \|f\|_\Omega^2,
\]
where, in the last step, we used elliptic regularity \eqref{eq: elliptic regularity} to write $\|u\|_{2, \Omega} \lesssim \|f\|_\Omega$ and $\|u_T^P\|_{2, T} \lesssim \|f\|_T$ for all $T \in \mesh$.

\begin{proof}[Proof of Theorem \ref{theorem: condensed energy error estimate}]
  By definition \eqref{eq:aNormGlob} of the condensed norm and linearity of $\recGlob^k$, it holds
  \begin{align*}
    \aNormGlob{\trialGlobTrace - \projLTraceGlob u} &= \aNormNonCondGlob{(\recGlob^k\trialGlobTrace-\recGlob^k\projLTraceGlob u, \trialGlobTrace-\projLTraceGlob u)}\\
    &= \aNormNonCondGlob{(\recGlob^k\trialGlobTrace + \vGlob f, \trialGlobTrace)-\interpolGlob{u} - (\recGlob^k\projLTraceGlob u + \vGlob f - \projLGlob u, 0)}\\
    &\le \aNormNonCondGlob{(\recGlob^k\trialGlobTrace + \vGlob f, \trialGlobTrace)-\interpolGlob{u}} + \aNormNonCondGlob{(\recGlob^k\projLTraceGlob u + \vGlob f - \projLGlob u, 0)},
  \end{align*}
  where we have added and subtracted the quantity $(\vGlob f, 0) + (\projLGlob u, 0)$ in the second step and used triangle inequalities to conclude.
  Recalling \eqref{eq: condensed and non-condendensed solutions relation} and \eqref{eq:errGlob}, the first term in the right-hand side of the above expression is precisely the error in the non-condensed energy norm, and it is bounded by \eqref{eq: theorem non-condensed energy error estimate}.
  For the second term, we substitute $u|_T \overset{\eqref{eq: u|_T = u_T^P + u_T^H}} = u_T^P + u_T^H$ for all $T \in \mesh$ to get:
  \[
  \begin{aligned}
    \aNormNonCondGlob{(\recGlob^k\projLTraceGlob u + \vGlob f - \projLGlob u, 0)}^2
    &= \sumT \aNormNonCondLoc{(\recLoc^k\projLTraceLoc u + \vLoc f - \projLLoc u_T^P - \projLLoc u_T^H, 0)}^2\\
    &\le 2 \sumT\left(
    \aNormNonCondLoc{(\recLoc^k\projLTraceLoc u - \projLLoc u_T^H, 0)}^2 + \aNormNonCondLoc{(\vLoc f - \projLLoc u_T^P, 0)}^2
    \right)
    \\
    \overset{\eqref{eq: U_T^k energy error estimate}, \eqref{eq: V_T^k energy error estimate}}
    &\lesssim \errFac^{2}\sumT\left(
    \|u_T^H\|^2_{2, T} + \|f\|^2_{T}
    \right)
    \\
    \overset{\eqref{eq: u_T^H help bound}}
    &\lesssim \errFac^2\|f\|_\Omega^2.\qedhere
  \end{aligned}
  \]
\end{proof}

\begin{proof}[Proof of Theorem \ref{theorem: condensed L2-error estimate}]
  Let $\errGlobTrace \coloneqq \trialGlobTrace - \projLTraceGlob u \in \hhoSpaceTrace_h^k$.
  We proceed as follows:
  \begin{equation}
    \label{eq: condensed L2-estimate approach}
    \begin{aligned}
      \|\errGlobTrace\|^2_{\facesUnion} &\le \sumT\|\errLocTrace\|_{\partial T}^2
      \\
      &\le 2 \sumT\left(
      \|\errLocTrace - \recLoc^k\errLocTrace\|_{\partial T}^2 + \|\recLoc^k\errLocTrace\|_{\partial T}^2
      \right)
      \\
      \overset{\eqref{eq: stabilization: trace - cell dof bound}, \eqref{eq:aNormGlob}, \eqref{eq: multiplicative trace inequality}}
      &\lesssim \sumT\left(
      \frac{h_T}{k}\aNormLoc{\errLocTrace}^2 + \|\recLoc^k\errLocTrace\|_T\|\nabla\recLoc^k\errLocTrace\|_T + h_T^{-1}\|\recLoc^k\errLocTrace\|_T^2
      \right)
      \\
      \overset{\eqref{eq: bound nabla cell dof in terms of energy norm}, \eqref{eq: mesh quasi uniformity assumption}}
      &\lesssim \frac{h}{k}\aNormGlob{\errGlobTrace}^2 + \sumT\|\recLoc^k\errLocTrace\|_T k^{\frac12} \aNormNonCondLoc{(\recLoc^k \errLocTrace, \errLocTrace)} + h^{-1}\|\recGlob^k\errGlobTrace\|_\Omega^2
      \\
      &\le \frac{h}{k}\aNormGlob{\errGlobTrace}^2 + k^{\frac12}\left(
      \sumT \|\recLoc^k\errLocTrace\|_T^2
      \right)^{\frac{1}{2}} \left(
      \sumT \aNormNonCondLoc{(\recLoc^k \errLocTrace, \errLocTrace)}^2 
      \right)^{\frac{1}{2}} + h^{-1}\|\recGlob^k\errGlobTrace\|_\Omega^2
      \\
      \overset{\eqref{eq:aNormGlob}}&= \frac{h}{k}\aNormGlob{\errGlobTrace}^2
      + k^{\frac12}\|\recGlob^k\errGlobTrace\|_\Omega \aNormGlob{\errGlobTrace}
      + h^{-1}\|\recGlob^k\errGlobTrace\|_\Omega^2,
    \end{aligned}
  \end{equation}
  where we used the inequality $(a + b)^2 \le 2 (a^2 + b^2)$ after adding and subtracting the local element component $\recLoc^k\errLocTrace$ in the second step, while,
  in the penultimate step, we used a Cauchy--Schwarz inequality on the sums.
  The energy error $\aNormGlob{\errGlobTrace}$ is bounded by Theorem \ref{theorem: condensed energy error estimate}.
  Thus, it only remains to bound the factor $\|\recGlob^k\errGlobTrace\|_\Omega = \|\recGlob^k\trialGlobTrace - \recGlob^k\projLTraceGlob u\|_\Omega$.
  Adding and subtracting the quantity $\vGlob f - \projLGlob u$,
  we get
  \begin{align*}
    \|\recGlob^k\errGlobTrace\|_\Omega \le \|\recGlob^k\trialGlobTrace + \vGlob f - \projLGlob u\|_\Omega + \|\recGlob^k\projLTraceGlob u +  \vGlob f - \projLGlob u\|_\Omega.
  \end{align*}
  Recalling \eqref{eq: condensed and non-condendensed solutions relation}, the first term in the right-hand side of the above expression is precisely the element $L^2$-error, which is bounded by \eqref{eq: theorem non-condensed L2 error estimate}.
  For the second term, we substitute $u|_T \overset{\eqref{eq: u|_T = u_T^P + u_T^H}}= u_T^P + u_T^H$ for all $T \in \mesh$ to write
  \[
  \begin{aligned}
    \|\recGlob^k\projLTraceGlob u +  \vGlob f - \projLGlob u\|_{\Omega}^2
    &= \sumT \|\recLoc^k\projLTraceLoc u +  \vLoc f - \projLLoc u_T^P - \projLLoc u_T^H\|_{T}^2
    \\
    &\le 2 \sumT \left(
    \|\recLoc^k\projLTraceLoc u - \projLLoc u_T^H\|_{T}^2 + \|\vLoc f - \projLLoc u_T^P\|_{T}^2
    \right)
    \\
    \overset{\eqref{eq: U_T^k L2 error estimate}, \eqref{eq: V_T^k L2 error estimate}}
    &\lesssim \errFac^{4}\sumT\left(
    \|u_T^H\|^2_{2, T} + \|f\|^2_{T}
    \right)
    \overset{\eqref{eq: u_T^H help bound}}\lesssim \errFac^{4}\|f\|_\Omega^2.\qedhere
  \end{aligned}
  \]
\end{proof}


\section{Multigrid analysis} \label{section: multigrid analysis}

Theorem \ref{theorem: convergence bound non-inherited W-cycle} will follow if the coarse problem approximates the fine problem sufficiently well. This is formalized by the \emph{regularity and approximation} criterion.
Recall the multigrid notation of Section \ref{section: multigrid algorithms}. For all $j=2,\ldots,\nLevel$, we define the operator $\projMGJMinus : \hhoSpaceTrace_h^{k_{j}} \rightarrow \hhoSpaceTrace_h^{k_{j-1}}$ such that, for all $\trialGlobTrace\in \hhoSpaceTrace_h^{k_{j}}$,
\begin{equation}
  \label{eq: P_kj for multigrid definition}
  a_h^{k_{j-1}}(\projMGJMinus \trialGlobTrace, \testGlobTrace) = a_h^{k_{j}}(\trialGlobTrace, \testGlobTrace) \qquad \forall \testGlobTrace\in \hhoSpaceTrace_h^{k_{j-1}}.
\end{equation}
The regularity and approximation criterion states that there exists $C>0$,
such that, for all $j=2,\ldots,\nLevel$ and all $\trialGlobTrace\in\hhoSpaceTrace_h^{k_j}$,
\begin{equation}
  \label{eq: reg and approx crit}
  |
  a_h^{k_j}\big(\trialGlobTrace - \projMGJMinus\trialGlobTrace, \trialGlobTrace\big)
  | \le C\frac{\| A_h^{k_j} \trialGlobTrace\|_{\facesUnion}^2}{\lambda^{k_j}},
\end{equation}
where $\lambda^{k_j}$ denotes the largest eigenvalue of $A_h^{k_j}$.

\begin{lem}[Approximation criterion]
  \label{lemma: reg and approx for non-inherited}
  For all $j=2,\ldots, \nLevel$, set
  \begin{equation}\label{eq: jump degrees}
    \sigma_j \coloneqq \frac{k_j}{k_{j-1}}.
  \end{equation}
  Then, there exists $\widetilde{C} > 0$ independent of $h, \nLevel$, and the $k_j$ such that criterion \eqref{eq: reg and approx crit} holds with $C = \max_j(\sigma_j^{\frac{5}{4}}k_j^{\frac{7}{4}})\widetilde{C}$.
\end{lem}

\begin{proof}[Proof of Theorem \ref{theorem: convergence bound non-inherited W-cycle}]
  Assuming Lemma \ref{lemma: reg and approx for non-inherited}, the result follows applying \cite[Theorem 7]{Bramble.Pasciak.ea:91}.
\end{proof}

The rest of this section is devoted to proving Lemma \ref{lemma: reg and approx for non-inherited}.

\subsection{Estimate of the largest eigenvalue}

\begin{lem}[Estimate of the largest eigenvalue]
  \label{lemma: largest eigenvalues for multigrid}
  For any $j=1\ldots,\nLevel$, we have
    \begin{equation}\label{eq: largest eigenvalues for multigrid}
      \lambda^{k_j} \lesssim \frac{k_j^2}{h}.
    \end{equation}
\end{lem}

\begin{proof}
  For any $\trialGlobTrace\in\hhoSpaceTrace_h^{k_j}$, we bound $\|\trialGlobTrace\|_{a_h^{k_j}}^2$ (see \eqref{eq:aNormGlob}) in terms of $\|\trialGlobTrace\|_\facesUnion^2$. We have, for all $T \in \mesh$,
  \[
  \begin{aligned}
    \aNormNonCondLocM{(\recLoc^{k_j}\trialLocTrace, 0)}^2
    \overset{\eqref{eq:non condensed norms}}&=
    \aNonCondLocArgM{(\recLoc^{k_j}\trialLocTrace, 0)}{(\recLoc^{k_j}\trialLocTrace, 0)}
    \\
    \overset{\eqref{eq: U_T^k definition}}&=
      -\aNonCondLocArgM{(0, \trialLocTrace)}{(\recLoc^{k_j}\trialLocTrace, 0)}
      \le \aNormNonCondLocM{(0, \trialLocTrace)}
      \aNormNonCondLocM{(\recLoc^{k_j}\trialLocTrace, 0)},
  \end{aligned}
  \]
  where we applied a Cauchy--Schwarz inequality on $\aNonCondLocM$ to conclude.
  Simplifying, this yields $\aNormNonCondLocM{(\recLoc^{k_j}\trialLocTrace, 0)} \le \aNormNonCondLocM{(0, \trialLocTrace)}$,
  and, squaring, summing over $T \in \mesh$, and then taking the square root of the resulting inequality, $\aNormNonCondGlobM{(\recGlob^{k_j}\trialGlobTrace, 0)} \le \aNormNonCondGlobM{(0, \trialGlobTrace)}$.
  Hence, we can write
  \begin{equation}\label{eq: largest eigenvalue:basic}
    \begin{aligned}
      \aNormGlobM{\trialGlobTrace}
      \overset{\eqref{eq:aNormGlob}}&=
      \aNormNonCondGlobM{(\recGlob^{k_j} \trialGlobTrace, \trialGlobTrace)}
      \\
      &\le
      \aNormNonCondGlobM{(\recGlob^{k_j} \trialGlobTrace, 0)} + \aNormNonCondGlobM{(0, \trialGlobTrace)}
      \\
      &\lesssim \aNormNonCondGlobM{(0, \trialGlobTrace)}
      = \left(
      \sumT \aNormNonCondLocM{(0, \trialLocTrace)}^2
      \right)^{\frac12},
    \end{aligned}
  \end{equation}
  where we used a triangle inequality in the second step.
  Thus, it is enough, for any $T\in\mesh$, to bound
  \begin{equation}\label{eq: largest eigenvalue.aNormNonCondLoc}
    \aNormNonCondLocM{(0, \trialLocTrace)}^2 \overset{\eqref{eq:non condensed norms}, \eqref{eq:aNonCondGlob}}= \|\gradHorNoDeg^{k_j}(0, \trialLocTrace)\|_T^2 + \sNormNonCondLocM{(0, \trialLocTrace)}^2.
  \end{equation}
  We have
  \begin{align*}
    \|\gradHorNoDeg^{k_j}(0, \trialLocTrace)\|_T^2
    \overset{\eqref{eq: gradHor definition}} = \big(\trialLocTrace, \gradHorNoDeg^{k_j}(0, \trialLocTrace)\cdot\normalT\big)_{\partial T} \le \|\trialLocTrace\|_{\partial T}\|\gradHorNoDeg^{k_j}(0, \trialLocTrace)\|_{\partial T} \overset{\eqref{eq: discrete inverse trace inequality}} \lesssim \|\trialLocTrace\|_{\partial T} \frac{k_j}{h_T^{\frac{1}{2}}}\|\gradHorNoDeg^{k_j}(0, \trialLocTrace)\|_{T},
  \end{align*}
  where we used a Cauchy--Schwarz inequality in the second step.
  Simplifying and squaring, we get
  \begin{equation}\label{eq: largest eigenvalue.estimate G}
    \|\gradHorNoDeg^{k_j}(0, \trialLocTrace)\|_T^2\lesssim \frac{k_j^2}{h_T} \|\trialLocTrace\|_{\partial T}^2.
  \end{equation}
  For the second term in the right-hand side of \eqref{eq: largest eigenvalue.aNormNonCondLoc}, we write
  \begin{equation}\label{eq: largest eigenvalue.estimate stab}
    \begin{aligned}
      \sNormNonCondLocM{(0, \trialLocTrace)}^2
      \overset{\eqref{eq: standard HHO stabilization}, \eqref{eq:difference operators}}
      &= \frac{k_j}{h_T}\|\trialLocTrace - \projLTraceLocM\horNoDeg^{k_j + 1}(0, \trialLocTrace) + \projLLocM\horNoDeg^{k_j + 1}(0, \trialLocTrace)\|_{\partial T}^2
      \\
      &\lesssim \frac{k_j}{h_T}\left(
      \|\trialLocTrace\|_{\partial T}^2 + \|\projLTraceLocM \big(
      \horNoDeg^{k_j + 1}(0, \trialLocTrace) - \projLLocM\horNoDeg^{k_j + 1}(0, \trialLocTrace)
      \big)
      \|_{\partial T}^2
      \right)
      \\
      \overset{\eqref{eq: stability of L2-projection on local traces}, \eqref{eq: approximation L2-projection on trace FIRST ORDER}}
      &\lesssim \frac{k_j}{h_T}\left(
      \|\trialLocTrace\|_{\partial T}^2 + \frac{h_T}{k_j}\|\nabla\horNoDeg^{k_j + 1}(0, \trialLocTrace)\|_{T}^2
      \right)
      \\
      \overset{\eqref{eq: L2-norm hor <= L2-norm gradHor}}
      &\le \frac{k_j}{h_T}\|\trialLocTrace\|_{\partial T}^2 + \|\gradHorNoDeg^{k_j}(0, \trialLocTrace)\|_{T}^2
      \overset{\eqref{eq: largest eigenvalue.estimate G}}\lesssim \frac{k_j^2}{h_T}\|\trialLocTrace\|_{\partial T}^2,
    \end{aligned}
  \end{equation}
  where we have additionally used the fact that $k_j \ge 1$ in the conclusion.
  Using \eqref{eq: largest eigenvalue.estimate G} and \eqref{eq: largest eigenvalue.estimate stab} to estimate the right-hand side of \eqref{eq: largest eigenvalue.aNormNonCondLoc}, summing the resulting inequality over $T \in \mesh$, and plugging it into \eqref{eq: largest eigenvalue:basic}, we finally get
  \[
  \aNormGlobM{\trialGlobTrace}^2
  \lesssim \sum_{T \in \mesh} \frac{k_j^2}{h_T}\|\trialLocTrace\|_{\partial T}^2
  \overset{\eqref{eq: mesh quasi uniformity assumption}}{
    \lesssim} \frac{k_j^2}{h}\|\trialGlobTrace\|_\facesUnion^2. \qedhere
  \]
\end{proof}

\subsection{Derivation of the approximation criterion} \label{section: multigrid analysis non-inherited W-cycle}

In light of Lemma \ref{lemma: largest eigenvalues for multigrid},  Lemma \ref{lemma: reg and approx for non-inherited} is proved if, for any $j\in\{2,\ldots,\nLevel\}$, it holds, for all $\trialGlobTrace\in\hhoSpaceTrace_h^{k_j}$,
\begin{equation}
  \label{eq: desired bound for algorithm 1}
  |a_h^{k_j}\big(\trialGlobTrace - \projMGJMinus\trialGlobTrace, \trialGlobTrace\big)| \lesssim \frac{\sigma_j^{\frac{5}{4}}h}{k_j^{\frac{1}{4}}} \|A_h^{k_j}\trialGlobTrace\|_{\facesUnion}^2,
\end{equation}
as then
\[
|a_h^{k_j}\big(\trialGlobTrace - \projMGJMinus\trialGlobTrace, \trialGlobTrace\big)|
\overset{\eqref{eq: largest eigenvalues for multigrid}}\lesssim
\sigma_j^{\frac{5}{4}}k_j^{\frac{7}{4}} \frac{\|A_h^{k_j}\trialGlobTrace\|_{\facesUnion}^2}{\lambda^{k_j}}.
\]
and \eqref{eq: reg and approx crit} will be satisfied with $C = \max_j(\sigma_j^{\frac{5}{4}} k_j^{\frac{7}{4}})\widetilde{C}$, where $\widetilde{C}$ has the same dependencies as the hidden constant in $\lesssim$ (see Section \ref{sec: setting: inequalities}).
For any $\trialGlobTrace\in\hhoSpaceTrace_h^{k_j}$, we shall introduce an auxiliary element $\tildeTrialGlobTrace\in\hhoSpaceTrace_h^{k_j}$ which is close to $\trialGlobTrace$ but which solves an appropriate Poisson problem. Then, the error analysis of Section \ref{section: hho main results} can be used to bound the left hand-side of \eqref{eq: desired bound for algorithm 1}.

Let $\lDeg\ge 1$ be a polynomial degree. Recall the simplicial submesh $(\submesh, \subfaces)$ of $(\mesh, \faces)$ from Section \ref{section: mesh}.
For any $S\in\submesh$, we set $\mathcal{G}_S \coloneqq \{G\in\subfaces : G \subset \partial S\}$ and we define $\Pi_{\mathcal{G}_S}^\lDeg : L^2(\partial S)\rightarrow \IP^\lDeg(\mathcal{G}_S)$ (with $\IP^\lDeg(\mathcal{G}_S)$ broken polynomial space of degree $\lDeg$ on $\partial T$) such that, for any $v\in L^2(\partial S)$,
\[
(\Pi_{\mathcal{G}_S}^\lDeg v, w )_{\partial S} = (v, w)_{\partial S} \qquad \forall w\in \IP^\lDeg(\mathcal{G}_S).
\]
We define an extension $\E: \hhoSpaceTrace_h^\lDeg \rightarrow \IP^\lDeg(\subfaces)$ by setting, for all $\testGlobTrace\in \hhoSpaceTrace_h^\lDeg$ and all $G\in\subfaces$,
\begin{equation}\label{eq: extension of traces}
  \E\testGlobTrace|_G \coloneqq
  \begin{cases}
    \testGlobTrace|_G & \text{if there exists $F \in \faces$ such that $G\subseteq F$},
    \\
    0 & \text{otherwise.}
  \end{cases}
\end{equation}

\begin{lem}[Raviart--Thomas lifting]
  \label{lemma: Raviart--Thomas element FMuLoc}
  For any $\trialGlobTrace\in\hhoSpaceTrace_h^{\lDeg}$, there exists a vector field $\FMuGlob{\lDeg+1}\in\IP^{\lDeg + 1}(\submesh)^d$ such that, for all $S\in\submesh$, and letting $\FMuLoc{\lDeg+1} \coloneqq \FMuGlob{\lDeg+1}|_S$,
  \begin{subequations}\label{eq: FMuLoc}
    \begin{alignat}{2}\label{eq: FMuLoc condition in S}
      (\FMuLoc{\lDeg+1}, \nabla v_S)_S &= 0 &\quad& \forall v_S\in\IP^\lDeg(S),
      \\ \label{eq: FMuLoc condition on the boundary of S}
      (\FMuLoc{\lDeg+1} \cdot\normalS, \nu_S)_{\partial S} &= \frac{1}{2}\big(\E A_h^\lDeg\trialGlobTrace, \nu_S\big)_{\partial S} &\quad& \forall \nu_S\in\IP^\lDeg(\mathcal{G}_S).
    \end{alignat}
  \end{subequations}
\end{lem}

\begin{proof}
  We can, for all $S\in\submesh$, define $\FMuLoc{\lDeg+1}\in\IP^{\lDeg + 1}(S)^d$ to be the local Raviart--Thomas element of polynomial degree $\lDeg+1$ (see \cite{Raviart.Thomas:77,Nedelec:80} and also \cite{Gatica:14}), which is uniquely defined by the conditions
  \[
  \boldsymbol \Pi_S^{\lDeg - 1}\FMuLoc{\lDeg+1} = 0,\qquad
  \Pi_{\mathcal{G}_S}^\lDeg\FMuLoc{\lDeg+1}\cdot\normalS =
  \frac12 (\E A_h^\lDeg\trialGlobTrace)|_{\partial S}.\qedhere
  \]
\end{proof}

  \begin{rem}[Global regularity of the lifting]
    Notice that the lifting $\FMuGlob{\lDeg+1}$ defined inside each $S \in \submesh$ by \eqref{eq: FMuLoc} is not in $H(\operatorname{div}; \Omega)$ because, although the right-hand side of condition \eqref{eq: FMuLoc condition on the boundary of S} is single-valued on the mesh skeleton, this condition prescribes the normal component of $\FMuLoc{\lDeg+1}$ relative to $S$ (and not to each simplicial face).
  \end{rem}

Let $\nabla_h\cdot : H^1(\submesh)^d\rightarrow L^2(\Omega)$ denote the broken divergence operator on $\submesh$, which restricted to each $S\in\submesh$ acts as the usual divergence operator.

\begin{lem}[Approximation properties of $\tildeTrialGlobTrace$]
  Given $\trialGlobTrace\in\hhoSpaceTrace_h^{\lDeg}$, $\lDeg \ge 1$, let $\tildeTrialGlobTrace\in\hhoSpaceTrace_h^{\lDeg}$ be the unique solution to the equation
  \begin{equation}
    \label{eq: aux variable tilde trial glob trace definition}
    a_h^\lDeg(\tildeTrialGlobTrace, \testGlobTrace) = (\nabla_h\cdot\FMuGlob{\lDeg+1}, \recGlob^\lDeg\testGlobTrace)_\Omega \qquad \forall \testGlobTrace\in \hhoSpaceTrace_h^\lDeg.
  \end{equation}
  Then, it holds:
  \begin{align}
    \label{eq: energy error aux rhs}
    \aNormGlobNoDeg{\lDeg}{\trialGlobTrace - \tildeTrialGlobTrace} \lesssim \errFacL^{\frac{1}{2}}\|A_h^\lDeg\trialGlobTrace\|_\facesUnion,\\
    \label{eq: L2 bound aux rhs}
    \|\nabla_h\cdot\FMuGlob{\lDeg+1}\|_\Omega \lesssim \frac{\lDeg}{h^{\frac{1}{2}}}\|A_h^\lDeg\trialGlobTrace\|_\facesUnion.
  \end{align}
\end{lem}

\begin{proof}
  For any $S\in\submesh$, let $T_S\in\mesh$ denote the unique element such that $S\subseteq T_S$.
  \medskip\\
  \emph{Proof of \eqref{eq: energy error aux rhs}.} We write, for any $\testGlobTrace \in \hhoSpaceTrace_h^\lDeg$,
  \[
  \begin{aligned}
    a_h^\lDeg(\trialGlobTrace - \tildeTrialGlobTrace, \testGlobTrace)
    \overset{\eqref{eq: operator corresponding to general bilinear form level j}, \eqref{eq: aux variable tilde trial glob trace definition}}
    &= (A_h^\lDeg\trialGlobTrace, \testGlobTrace)_\facesUnion - (\nabla_h\cdot\FMuGlob{\lDeg+1}, \recGlob^\lDeg\testGlobTrace)_\Omega\\
    &= (A_h^\lDeg\trialGlobTrace, \testGlobTrace)_\facesUnion - \sumS(\nabla\cdot\FMuLoc{\lDeg+1}, \recLocS^\lDeg\testLocTraceS)_S\\
    &= (A_h^\lDeg\trialGlobTrace, \testGlobTrace)_\facesUnion - \sumS\left[
      -(\FMuLoc{\lDeg+1}, \nabla\recLocS^\lDeg\testLocTraceS)_S + (\FMuLoc{\lDeg+1}\cdot\normalS, \recLocS^\lDeg\testLocTraceS)_{\partial S}
      \right]
    \\
    \overset{\eqref{eq: FMuLoc}}
    &= (A_h^\lDeg\trialGlobTrace, \testGlobTrace)_\facesUnion - \frac{1}{2}\sumS (\E A_h^\lDeg\trialGlobTrace, \recLocS^\lDeg\testLocTraceS)_{\partial S}
    \\
    \overset{\eqref{eq: extension of traces}}
    &= \frac{1}{2} \sumT \left[
      (A_h^\lDeg\trialGlobTrace, \testLocTrace)_{\partial T} - (A_h^\lDeg\trialGlobTrace, \recLoc^\lDeg\testLocTrace)_{\partial T}
      \right]
    \\
    \overset{\eqref{eq: stabilization: trace - cell dof bound}}
    &\lesssim \left(
    \sumT\|A_h^\lDeg\trialGlobTrace\|_{\partial T}^2
    \right)^{\frac{1}{2}} \left(
    \sumT\frac{h_T}{\lDeg}\aNormNonCondLocNoDeg{\lDeg}{(\recLoc^\lDeg\testLocTrace, \testLocTrace)}^2
    \right)^{\frac{1}{2}}
    \\
    &\lesssim \errFacL^{\frac{1}{2}}\|A_h^\lDeg\trialGlobTrace\|_{\facesUnion}\aNormGlobNoDeg{\lDeg}{\testGlobTrace},
  \end{aligned}
  \]
  where in the third step we integrated by parts on each $S\in\submesh$, in the fifth step we have additionally used the fact that $\testGlobTrace$ vanishes on boundary faces, and
  in the sixth step we used a Cauchy--Schwarz inequality on the sum. Choosing $\testGlobTrace = \trialGlobTrace - \tildeTrialGlobTrace$ and simplifying completes the proof.
  \medskip\\
  \emph{Proof of \eqref{eq: L2 bound aux rhs}.} We
  proceed by a similar local expansion and integration by parts:
  \begin{align*}
    \|\nabla_h\cdot\FMuGlob{\lDeg+1}\|_\Omega^2
    &= \sumS(\nabla\cdot\FMuLoc{\lDeg+1}, \nabla\cdot\FMuLoc{\lDeg+1})_S
    \\
    &= \sumS\left[
      -\big(\FMuLoc{\lDeg+1}, \nabla(\nabla\cdot\FMuLoc{\lDeg+1})\big)_S + (\FMuLoc{\lDeg+1}\cdot\normalS, \nabla\cdot\FMuLoc{\lDeg+1})_{\partial S}
      \right]
    \\
    \overset{\eqref{eq: FMuLoc}}
    &= \frac{1}{2}\sumS(\E A_h^\lDeg\trialGlobTrace, \nabla\cdot\FMuLoc{\lDeg+1})_{\partial S}
    \\
    &\le \left(
    \sumS \|\E A_h^\lDeg\trialGlobTrace\|_{\partial S}^2
    \right)^{\frac{1}{2}}\left(
    \sumS \|\nabla\cdot\FMuLoc{\lDeg+1}\|_{\partial S}^2
    \right)^{\frac{1}{2}}
    \\
    \overset{\eqref{eq: extension of traces}}
    &\lesssim \left(
    \sumT \|A_h^\lDeg\trialGlobTrace\|_{\partial T}^2
    \right)^{\frac{1}{2}} \left(
    \sumS \frac{\lDeg^2}{h_S}\|\nabla\cdot\FMuLoc{\lDeg+1}\|_{S}^2
    \right)^{\frac{1}{2}}
    \\
    &\lesssim \frac{\lDeg}{h^{\frac{1}{2}}}\|A_h^\lDeg\trialGlobTrace\|_{\facesUnion}\|\nabla_h\cdot\FMuGlob{\lDeg+1}\|_\Omega,
  \end{align*}
  where in the fifth step we additionally used a trace inequality analogous to \eqref{eq: discrete inverse trace inequality} but on the submesh,
  in the penultimate step we used a Cauchy--Schwarz inequality on the sum,
  and in the last step $h_T\lesssim h_S$ (consequence of mesh regularity) together with the mesh quasi-uniformity assumption. We conclude after simplification.
\end{proof}

\begin{prop}[Basic regularity and approximation estimate]
  The estimate \eqref{eq: desired bound for algorithm 1} holds.
\end{prop}

\begin{proof}
  For any $\trialGlobTrace\in\hhoSpaceTrace_h^{\mDeg}$, we have
  \begin{equation}
    \label{eq: first step}
    a_h^{\mDeg}\big(\trialGlobTrace-\projMGN\trialGlobTrace,\trialGlobTrace)\overset{\eqref{eq: operator corresponding to general bilinear form level j}}=
    (\trialGlobTrace-\projMGN\trialGlobTrace, A_h^{\mDeg} \trialGlobTrace)_\facesUnion
    \le \|\trialGlobTrace-\projMGN\trialGlobTrace\|_\facesUnion\|A_h^{\mDeg} \trialGlobTrace\|_\facesUnion,
  \end{equation}
  where we additionally used the symmetry of $a_h^{\mDeg}$ in the first step.
  Let now $\FMuGlob{\mDeg+1}\in\IP^{\mDeg+1}(\submesh)^d$ and $\FPNMuGlob{\nDeg+1}\in\IP^{\nDeg+1}(\submesh)^d$
  be defined as in Lemma \ref{lemma: Raviart--Thomas element FMuLoc} for $\trialGlobTrace\in\hhoSpaceTrace_h^{\mDeg}$ and $\projMGN\trialGlobTrace\in\hhoSpaceTrace_h^{\nDeg}$, respectively.
  Let also $\tildeTrialGlobTrace\in\hhoSpaceTrace_h^{\mDeg}$ and $\tildePNTrialGlobTrace\in\hhoSpaceTrace_h^{\nDeg}$ be the corresponding auxiliary fields defined as in \eqref{eq: aux variable tilde trial glob trace definition}.
  Applying a triangle inequality to the first factor in the right hand-side of \eqref{eq: first step} after inserting $\pm (\tildeTrialGlobTrace - \tildePNTrialGlobTrace)$, we get
  \[
  \|\trialGlobTrace-\projMGN\trialGlobTrace\|_\facesUnion
  \le
  \|\tildeTrialGlobTrace - \tildePNTrialGlobTrace\|_\facesUnion
  + \|\trialGlobTrace - \tildeTrialGlobTrace\|_\facesUnion
  + \|\projMGN\trialGlobTrace - \tildePNTrialGlobTrace\|_\facesUnion.
  \]
  We will prove that
  \begin{align}
    \label{eq: claimed bound 3}
    \|\tildeTrialGlobTrace - \tildePNTrialGlobTrace\|_\facesUnion & \lesssim \sigma_j^{\frac{5}{4}}\frac{h}{\mDeg^{\frac{1}{4}}}\|A_h^{\mDeg}\mu\|_{\facesUnion},
    \\
    \label{eq: claimed bound 1}
    \|\trialGlobTrace - \tildeTrialGlobTrace\|_\facesUnion &\lesssim \frac{h}{\mDeg^{\frac{5}{8}}}\|A_h^{\mDeg}\mu\|_{\facesUnion},
    \\
    \label{eq: claimed bound 2}
    \|\projMGN\trialGlobTrace - \tildePNTrialGlobTrace\|_\facesUnion &\lesssim \sigma_j^{\frac{5}{8}}\frac{h}{\mDeg^{\frac{5}{8}}}\|A_h^{\mDeg}\mu\|_{\facesUnion}.
  \end{align}
  The estimate \eqref{eq: desired bound for algorithm 1} is obtained using the above bounds in \eqref{eq: first step} and noticing that the first term has the worse estimate and hence dominates.
  \medskip\\
  \emph{Proof of \eqref{eq: claimed bound 3}.}
  We first notice that
  \begin{equation}
    \label{eq: A^k P^k is proj of A^n}
    \projLTraceGlobN A_h^{\mDeg}\trialGlobTrace = A_h^{\nDeg}\projMGN\trialGlobTrace,
  \end{equation}
  as, for any $\testGlobTrace\in\hhoSpaceTrace_h^{\nDeg} \subset \hhoSpaceTrace_h^{\mDeg}$,
  \[
  (A_h^{\mDeg}\trialGlobTrace, \testGlobTrace)_{\facesUnion}
  \overset{\eqref{eq: operator corresponding to general bilinear form level j}}
  = a_h^{\mDeg}(\trialGlobTrace, \testGlobTrace)
  \overset{\eqref{eq: P_kj for multigrid definition}}
  = a_h^{\nDeg}(\projMGN \trialGlobTrace, \testGlobTrace)
  \overset{\eqref{eq: operator corresponding to general bilinear form level j}}
  = (A_h^{\nDeg}\projMGN\trialGlobTrace, \testGlobTrace)_{\facesUnion}.
  \]
  As a consequence there holds, for any $T\in\mesh$, that
  \begin{equation}
    \label{eq: projection property of RT rhs}
    \projLLocN \big(\nabla_h\cdot \FMuGlob{\mDeg+1}\big) = \nabla_h\cdot\FPNMuGlob{\nDeg+1}|_T.
  \end{equation}
  This follows observing that, for any $v_T\in\IP^{\nDeg}(T)$, we have
  \begin{equation}
    \label{eq: proving projection property of RT rhs}
    (\nabla_h\cdot\FMuGlob{\mDeg+1}, v_T)_T
    = \sum_{S \in \submeshT} (\nabla\cdot\FMuGlob{\mDeg+1}, v_T)_S
    \overset{\eqref{eq: FMuLoc}}
    = \sum_{S \in \submeshT} \frac{1}{2} (\E A_h^{\mDeg}\trialGlobTrace, v_T)_{\partial S}
    \overset{\eqref{eq: extension of traces}}
    = \frac{1}{2} (A_h^{\mDeg}\trialGlobTrace, v_T)_{\partial T},
  \end{equation}
  where $\submeshT \coloneqq \left\{ S \in \submesh : S \subseteq T\right\}$ and in the second step we used integration by parts.
    Hence, for any $T \in \mesh$ and any $v_T\in\IP^{\nDeg}(T)$,
    \[
    (\nabla_h\cdot\FMuGlob{\mDeg+1}, v_T)_T
    \overset{\eqref{eq: proving projection property of RT rhs}}=\frac{1}{2} (A_h^{\mDeg}\trialGlobTrace, v_T)_{\partial T}
    \overset{\eqref{eq: A^k P^k is proj of A^n}}= \frac{1}{2} (A_h^{\nDeg}\projMGN\trialGlobTrace, v_T)_{\partial T}
    = (\nabla_h\cdot\FPNMuGlob{\nDeg+1}, v_T)_T,
    \]
    where the last step is a consequence of \eqref{eq: proving projection property of RT rhs} with $\mu_h$ replaced by $\projMGN\mu_h$.
  This finishes the proof of relation \eqref{eq: projection property of RT rhs}.
  But then the auxiliary fields $\tildeTrialGlobTrace$ and $\tildePNTrialGlobTrace$ are, respectively, solutions of:
  \begin{alignat}{2} \label{eq: tildePNTrialGlobTrace}
    a_h^{\nDeg}(\tildePNTrialGlobTrace, \testGlobTrace)
    &= (\nabla_h\cdot\FPNMuGlob{\nDeg+1}, \recGlob^{\nDeg}\testGlobTrace)_\Omega
    \overset{\eqref{eq: projection property of RT rhs}}
    = (\nabla_h\cdot\FMuGlob{\mDeg+1}, \recGlob^{\nDeg}\testGlobTrace)_\Omega
    &\qquad& \forall\testGlobTrace\in\hhoSpaceTrace_h^{\nDeg},
    \\ \label{eq: tildeTrialGlobTrace}
    a_h^{\mDeg}(\tildeTrialGlobTrace, \testGlobTrace) &= (\nabla_h\cdot\FMuGlob{\mDeg+1}, \recGlob^{\mDeg}\testGlobTrace)_\Omega
    &\qquad& \forall\testGlobTrace\in\hhoSpaceTrace_h^{\mDeg},
  \end{alignat}
  that is, they are discrete solutions of different degrees to the same right hand-side $\nabla_h\cdot\FMuGlob{\mDeg+1}$.
  Let now $\widetilde{u}\in H_0^1(\Omega)$ solve
  \[
  (\nabla\widetilde{u}, \nabla v)_\Omega = (\nabla_h\cdot\FMuGlob{\mDeg+1}, v)_\Omega \qquad \forall v\in H_0^1(\Omega).
  \]
  Applying a triangle inequality after inserting $\pm (\projLTraceGlobArbitraryDeg^{\mDeg} \widetilde{u} - \projLTraceGlobArbitraryDeg^{\nDeg} \widetilde{u})$ into the norm yields
  \begin{align*}
    \|\tildeTrialGlobTrace - \tildePNTrialGlobTrace\|_\facesUnion
    &\le \|\tildeTrialGlobTrace - \projLTraceGlobArbitraryDeg^{\mDeg} \widetilde{u}\|_\facesUnion + \|\tildePNTrialGlobTrace - \projLTraceGlobArbitraryDeg^{\nDeg} \widetilde{u}\|_\facesUnion
    + \|\projLTraceGlobArbitraryDeg^{\mDeg} \widetilde{u} - \projLTraceGlobArbitraryDeg^{\nDeg} \widetilde{u}\|_\facesUnion\\
    &\le \|\tildeTrialGlobTrace - \projLTraceGlobArbitraryDeg^{\mDeg} \widetilde{u}\|_\facesUnion + \|\tildePNTrialGlobTrace - \projLTraceGlobArbitraryDeg^{\nDeg} \widetilde{u}\|_\facesUnion
    + \|\widetilde{u} - \projLTraceGlobArbitraryDeg^{\nDeg} \widetilde{u}\|_\facesUnion\\
    \overset{\eqref{eq: theorem condensed L2 error estimate}, \eqref{eq: approximation L2-projection on trace}}
    &\lesssim \left[
      \mDeg^{\frac14}\errFacM^{\frac{3}{2}} + \nDeg^{\frac14}\errFacN^{\frac{3}{2}} + \errFacN^{\frac{3}{2}}
      \right]  \left(
    \|\widetilde{u}\|_{2, \Omega} + \|\nabla_h\cdot\FMuGlob{\mDeg+1}\|_\Omega
    \right)
    \\
    \overset{\eqref{eq: elliptic regularity}}
    &\lesssim \Big(\frac{\sigma_j}{\mDeg}\Big)^{\frac{5}{4}}h^{\frac{3}{2}} \|\nabla_h\cdot\FMuGlob{\mDeg+1}\|_\Omega\\
    \overset{\eqref{eq: L2 bound aux rhs}}
    &\lesssim \Big(\frac{\sigma_j}{\mDeg}\Big)^{\frac{5}{4}}h^{\frac{3}{2}}\frac{\mDeg}{h^{\frac{1}{2}}} \|A_h^{\mDeg}\trialGlobTrace\|_\facesUnion,
  \end{align*}
  where in the penultimate step we used $\mDeg \overset{\eqref{eq: jump degrees}}= \sigma_j \nDeg$ together with $1 \le \nDeg$ for the last term in square brackets.
  \medskip\\
  \emph{Proof of \eqref{eq: claimed bound 1}.}
  By \eqref{eq: energy error aux rhs}, we can bound $\aNormGlobM{\trialGlobTrace - \tildeTrialGlobTrace}$.
  The $L^2(\facesUnion)$-bound for this quantity follows from a mix of arguments already used in the proofs of Theorems \ref{theorem: condensed energy error estimate} and \ref{theorem: condensed L2-error estimate}.
  Let $\errGlobTrace \coloneqq \trialGlobTrace - \tildeTrialGlobTrace \in \hhoSpaceTrace_h^{\mDeg}$.
  Copying the steps in \eqref{eq: condensed L2-estimate approach}, we obtain
  \begin{align}
    \label{eq: multigrid l2 error trace expression}
    \|\errGlobTrace\|^2_{\facesUnion}
    \lesssim \frac{h}{\mDeg}\aNormGlobM{\errGlobTrace}^2 +  \mDeg^{\frac{1}{2}}\|\recGlob^{\mDeg}\errGlobTrace\|_\Omega \aNormGlobM{\errGlobTrace} + h^{-1}\|\recGlob^{\mDeg}\errGlobTrace\|_\Omega^2.
  \end{align}
  Thus, it only remains to bound the factor $\|\recGlob^{\mDeg}\errGlobTrace\|_\Omega$. Let $z\in H_0^1(\Omega)$ be the elliptic lifting of $\recGlob^{\mDeg}\errGlobTrace$ obtained solving the following equation:
  \[
  (\nabla z, \nabla v)_\Omega = (\recGlob^{\mDeg}\errGlobTrace, v)_\Omega \qquad \forall v\in H^{1}_0(\Omega).
  \]
  Since $\recGlob^{\mDeg}\errGlobTrace\in L^2(\Omega)$, elliptic regularity yields $z\in H^2(\Omega)$ and $-\Delta z  = \recGlob^{\mDeg}\errGlobTrace$ almost everywhere on $\Omega$.
  Proceeding as in \eqref{eq: non-condensed L2-error error expression}, we obtain
  \begin{equation}
    \label{eq: multigrid error expression}
    \|\recGlob^{\mDeg}\errGlobTrace\|_\Omega^2
    \overset{\eqref{eq: cancellation via normal trace component}}
    = \sumT \Big[(\nabla \recLoc^{\mDeg}\errLocTrace, \projLLocVecM\nabla z)_T + (\errLocTrace-\recLoc^{\mDeg}\errLocTrace, \nabla z \cdot \normalT)_{\partial T}\Big].
  \end{equation}
  We will add and subtract two different expressions for $\aGlobArgM{\errGlobTrace}{\projLTraceGlobArbitraryDeg^{\mDeg} z}$ to the above error. The minus expression is, proceeding as in  \eqref{eq: standard derivation p.11},
  \begin{equation}
    \label{eq: multigrid minus expression}
    \begin{aligned}
      &\aGlobArgM{\errGlobTrace}{\projLTraceGlobArbitraryDeg^{\mDeg} z}
      \\
      &\begin{aligned}[t]
      	\overset{\eqref{eq:aNonCondGlob}}
      	&= \aNonCondGlobArgM{(\recGlob^{\mDeg} \errGlobTrace, \errGlobTrace)}{(\recGlob^{\mDeg}\projLTraceGlobArbitraryDeg^{\mDeg} z, \projLTraceGlobArbitraryDeg^{\mDeg} z)}\\
         \overset{\eqref{eq: U_T^k definition}}
         &= \aNonCondGlobArgM{(\recGlob^{\mDeg} \errGlobTrace, \errGlobTrace)}{\interpolGlobM{z}}\\
         \overset{\eqref{eq: gradHor commutation with interpolation}, \eqref{eq: gradHor definition}}
         &= \sumT\left[
           \big(\nabla\recLoc^{\mDeg} \errLocTrace, \projLLocVecM\nabla z\big)_T + (\errLocTrace-\recLoc^{\mDeg} \errLocTrace, \projLLocVecM\nabla z \cdot \normalT)_{\partial T} + \sNonCondLocArgM{(\recLoc^{\mDeg} \errLocTrace, \errLocTrace)}{\interpolLocM{z}}
           \right].
       \end{aligned}
    \end{aligned}
  \end{equation}
  To obtain the plus expression, for any $S\in\submesh$, let $T_S\in\mesh$ denote the unique element such that $S\subseteq T_S$.
  Expanding the error $\errGlobTrace = \trialGlobTrace - \tildeTrialGlobTrace$, using the definition \eqref{eq: tildeTrialGlobTrace} of $\tildeTrialGlobTrace$ with $\testGlobTrace = \projLTraceGlobArbitraryDeg^{\mDeg} z$, and integrating by parts on $S\in\submesh$, we get
  \begin{equation}
    \label{eq: multigrid plus expression 1}
    \begin{aligned}
      \aGlobArgM{\errGlobTrace}{\projLTraceGlobArbitraryDeg^{\mDeg} z}
      \overset{\eqref{eq: operator corresponding to general bilinear form level j}, \eqref{eq: aux variable tilde trial glob trace definition}}
      &= (A_h^{\mDeg}\trialGlobTrace, \projLTraceGlobArbitraryDeg^{\mDeg} z)_\facesUnion - (\nabla_h\cdot\FMuGlob{\mDeg+1}, \recGlob^{\mDeg}\projLTraceGlobArbitraryDeg^{\mDeg} z)_\Omega\\
      &= (A_h^{\mDeg}\trialGlobTrace, \projLTraceGlobArbitraryDeg^{\mDeg} z)_\facesUnion - \sumS\left[
        -(\FMuLoc{\mDeg+1}, \nabla\recLocS^{\mDeg} \projLTraceLocSArbitraryDeg^{\mDeg} z)_S
        + (\FMuLoc{\mDeg+1}\cdot\normalS, \recLocS^{\mDeg} \projLTraceLocSArbitraryDeg^{\mDeg} z)_{\partial S}
        \right]
      \\
      \overset{\eqref{eq: FMuLoc}}
      &= (A_h^{\mDeg}\trialGlobTrace, \projLTraceGlobArbitraryDeg^{\mDeg} z)_\facesUnion - \frac{1}{2}\sumS
        (\E A_h^{\mDeg}\trialGlobTrace, \recLocS^{\mDeg} \projLTraceLocSArbitraryDeg^{\mDeg} z)_{\partial S}
      \\
      \overset{\eqref{eq: extension of traces}}
      &= \frac{1}{2}\sumT (A_h^{\mDeg}\trialGlobTrace, \projLTraceLocArbitraryDeg^{\mDeg} z - \recLoc^{\mDeg} \projLTraceLocArbitraryDeg^{\mDeg} z)_{\partial T}\\
      &\lesssim \|A_h^{\mDeg}\trialGlobTrace\|_\facesUnion \left(
      \sumT \|\projLTraceLocArbitraryDeg^{\mDeg} z - \recLoc^{\mDeg} \projLTraceLocArbitraryDeg^{\mDeg} z\|_{\partial T}^2
      \right)^{\frac{1}{2}},
    \end{aligned}
  \end{equation}
  where in the last step we used a Cauchy--Schwarz inequality on the sum. For all $T\in\mesh$, let $z_T^H$ be defined by \eqref{eq: u_T^H} with $z$ replacing $w$ and let $z_T^P$ be defined by \eqref{eq: u_T^P definition} with $\recLoc^{\mDeg}\errLocTrace$ replacing $g$.
  Decomposing $z|_T$ as $z|_T = z_T^H + z_T^P$ according to \eqref{eq: u|_T = u_T^P + u_T^H}, we can write
  \begin{equation}
    \label{eq: z_T^H bound}
    \|z_T^H\|_{2, T} \le \|z_T^H + z_T^P\|_{2, T} + \|z_T^P\|_{2,T}
      \overset{\eqref{eq: elliptic regularity}}
      \lesssim \|z\|_{2,T} + \|\recLoc^{\mDeg}\errLocTrace\|_T.
  \end{equation}
  Substituting $\projLTraceLocArbitraryDeg^{\mDeg} z = \projLTraceLocArbitraryDeg^{\mDeg} z_T^H$ (since $z_T^P|_{\partial T} \equiv 0$) in the second factor of \eqref{eq: multigrid plus expression 1} yields:
  \[
  \begin{aligned}
    &\|\projLTraceLocArbitraryDeg^{\mDeg} z - \recLoc^{\mDeg} \projLTraceLocArbitraryDeg^{\mDeg} z\|_{\partial T}^2
    \\
    &\quad\begin{aligned}[t]
    & = \|\projLTraceLocArbitraryDeg^{\mDeg} z_T^H - \recLoc^{\mDeg} \projLTraceLocArbitraryDeg^{\mDeg} z\|_{\partial T}^2\\
    &\lesssim \|\projLTraceLocArbitraryDeg^{\mDeg} z_T^H - \projLLocM z_T^H\|_{\partial T}^2 + \|\projLLocM z_T^H -\recLoc^{\mDeg} \projLTraceLocArbitraryDeg^{\mDeg} z\|_{\partial T}^2\\
    \overset{\eqref{eq: stability of L2-projection on local traces}, \eqref{eq: multiplicative trace inequality}}
    &\lesssim \|z_T^H - \projLLocM z_T^H\|_{\partial T}^2 + \|\projLLocM z_T^H -\recLoc^{\mDeg} \projLTraceLocArbitraryDeg^{\mDeg} z\|_{T}\|\nabla (\projLLocM z_T^H - \recLoc^{\mDeg} \projLTraceLocArbitraryDeg^{\mDeg} z)\|_{T}
    \\
    &\quad + h_T^{-1}\|\projLLocM z_T^H -\recLoc^{\mDeg} \projLTraceLocArbitraryDeg^{\mDeg} z\|_{T}^2\\
    \overset{\eqref{eq: bound nabla cell dof in terms of energy norm}}
    &\lesssim \|z_T^H - \projLLocM z_T^H\|_{\partial T}^2 + \mDeg^{\frac{1}{2}}\|\projLLocM z_T^H -\recLoc^{\mDeg} \projLTraceLocArbitraryDeg^{\mDeg} z\|_{T}\aNormNonCondLoc{(\projLLocM z_T^H - \recLoc^{\mDeg} \projLTraceLocArbitraryDeg^{\mDeg} z, 0)}
    \\
    &\quad + h_T^{-1}\|\projLLocM z_T^H -\recLoc^{\mDeg} \projLTraceLocArbitraryDeg^{\mDeg} z\|_{T}^2.
    \end{aligned}
  \end{aligned}
  \]
    Using, \eqref{eq: approximation L2-projection on trace}, \eqref{eq: U_T^k L2 error estimate}, and \eqref{eq: U_T^k energy error estimate} to estimate the various norms appearing in the right-hand side of the above expression, we get
  \[
  \begin{aligned}
    \|\projLTraceLocArbitraryDeg^{\mDeg} z - \recLoc^{\mDeg} \projLTraceLocArbitraryDeg^{\mDeg} z\|_{\partial T}^2
    &\lesssim \errFacMLoc^{3}\|z_T^H\|_{2, T}^2 + \mDeg^{\frac{1}{2}}\errFacMLoc^3\|z_T^H\|_{2, T}^2 + \mDeg^{-1}\errFacMLoc^3\|z_T^H\|_{2, T}^2\\
    \overset{\eqref{eq: z_T^H bound}}
    &\lesssim \mDeg^{\frac{1}{2}}\errFacMLoc^3\left(\|z\|_{2, T}^2 + \|\recLoc^{\mDeg}\errLocTrace\|_T^2\right).
  \end{aligned}
  \]
    Substituting into \eqref{eq: multigrid plus expression 1},
    using the fact that $h_T \le h$ for all $T \in \mesh$,
    and invoking elliptic regularity to write
    \begin{equation}\label{eq:z.2.Omega.est}
      \|z\|_{2,\Omega}\overset{\eqref{eq: elliptic regularity}}\lesssim \|\recGlob^{\mDeg}\errGlobTrace\|_\Omega,
    \end{equation}
    we get
  \begin{equation}
    \label{eq: multigrid plus expression 2}
    \aGlobArgM{\errGlobTrace}{\projLTraceGlobArbitraryDeg^{\mDeg} z}
    \lesssim
    \mDeg^{\frac{1}{4}}\errFacM^{\frac{3}{2}} \|A_h^{\mDeg}\trialGlobTrace\|_\facesUnion
      \|\recGlob^{\mDeg}\errGlobTrace\|_\Omega.
  \end{equation}
  Subtracting \eqref{eq: multigrid minus expression} and adding \eqref{eq: multigrid plus expression 2} to \eqref{eq: multigrid error expression}, we obtain
  \[
  \begin{aligned}
    \|\recGlob^{\mDeg}\errGlobTrace\|_\Omega^2
    &\lesssim \sumT\left[
      \big(\errLocTrace-\recLoc^{\mDeg} \errLocTrace, (\nabla z - \projLLocVecM\nabla z) \cdot \normalT\big)_{\partial T} + \sNonCondLocArgM{(\recLoc^{\mDeg} \errLocTrace, \errLocTrace)}{\interpolLocM{z}}
      \right]
    \\
    &\quad
    + \mDeg^{\frac{1}{4}}\errFacM^{\frac{3}{2}} \|A_h^{\mDeg}\trialGlobTrace\|_\facesUnion  \|\recGlob^{\mDeg}\errGlobTrace\|_\Omega
    \\
    &\lesssim \left(
    \sumT\|\errLocTrace-\recLoc^{\mDeg} \errLocTrace\|_{\partial T}^2
    \right)^{\frac{1}{2}}\left(
    \sumT\|\nabla z - \projLLocVecM\nabla z\|_{\partial T}^2
    \right)^{\frac{1}{2}}
    \\
    &\quad
    + \left(
    \sumT \sNormNonCondLocM{(\recLoc^{\mDeg} \errLocTrace, \errLocTrace)}
    \right)^{\frac{1}{2}}\left(
    \sumT \sNormNonCondLocM{\interpolLocM{z}}
    \right)^{\frac{1}{2}}
    \\
    &\quad + \mDeg^{\frac{1}{4}}\errFacM^{\frac{3}{2}} \|A_h^{\mDeg}\trialGlobTrace\|_\facesUnion
    \|\recGlob^{\mDeg}\errGlobTrace\|_\Omega
    \\
    \overset{\eqref{eq: stabilization: trace - cell dof bound}, \eqref{eq: approximation L2-projection on trace}, \eqref{eq: stabilization consistency}}
    &\lesssim
    \frac{h}{\mDeg}\aNormNonCondGlobM{(\recGlob^{\mDeg} \errGlobTrace, \errGlobTrace)}\|z\|_{2,\Omega}
    + \mDeg^{\frac{1}{4}}\errFacM^{\frac{3}{2}} \|A_h^{\mDeg}\trialGlobTrace\|_\facesUnion
    \|\recGlob^{\mDeg}\errGlobTrace\|_\Omega
    \\
    \overset{\eqref{eq:aNormGlob}, \eqref{eq: energy error aux rhs}, \eqref{eq:z.2.Omega.est}}
    &\lesssim
    \errFacM^{\frac{3}{2}}\|A_h^{\mDeg}\trialGlobTrace\|_\facesUnion \|\recGlob^{\mDeg}\errGlobTrace\|_\Omega
    + \mDeg^{\frac{1}{4}}\errFacM^{\frac{3}{2}} \|A_h^{\mDeg}\trialGlobTrace\|_\facesUnion
    \|\recGlob^{\mDeg}\errGlobTrace\|_\Omega,
  \end{aligned}
  \]
  where in the second step we used Cauchy--Schwarz inequalities.
  Simplifying, we obtain $\|\recGlob^{\mDeg}\errGlobTrace\|_\Omega \lesssim \mDeg^{\frac{1}{4}}\errFacM^{\frac{3}{2}} \|A_h^{\mDeg}\trialGlobTrace\|_\facesUnion$. Substituting into \eqref{eq: multigrid l2 error trace expression} and again applying energy error estimate \eqref{eq: energy error aux rhs} finally yields
  \[
  \|\errGlobTrace\|^2_{\facesUnion}
  \lesssim
  \mDeg^{\frac{3}{4}}\errFacM^{2}\|A_h^{\mDeg}\trialGlobTrace\|_\facesUnion^2.
  \]
  \medskip\\
  \emph{Proof of \eqref{eq: claimed bound 2}.}
  Copying the derivation for the bound on $\|\trialGlobTrace - \tildeTrialGlobTrace\|_\facesUnion$, we have
  \[
  \|\projMGN\trialGlobTrace - \tildePNTrialGlobTrace\|_\facesUnion \lesssim \frac{h}{\nDeg^\frac{5}{8}}\|A_h^{\nDeg}\projMGN\trialGlobTrace\|_\facesUnion.
  \]
  We conclude after using $\|A_h^{\nDeg}\projMGN\trialGlobTrace\|_\facesUnion\overset{\eqref{eq: A^k P^k is proj of A^n}} \le \|A_h^{\mDeg}\trialGlobTrace\|_\facesUnion$ and $\mDeg \overset{\eqref{eq: jump degrees}}= \sigma_j \nDeg$.
\end{proof}


\section*{Acknowledgements}

The authors acknowledge the support of the French \emph{Agence Nationale de la Recherche} under grant ANR-23-CE46-0013 HIPOTHEC.


\printbibliography

\end{document}